\documentclass[12pt]{amsart}

\usepackage{amscd,latexsym,amsthm,amsfonts,amssymb,amsmath,amsxtra}
\usepackage[mathscr]{eucal}
\usepackage{tkz-euclide}
\usepackage{pst-plot}
\usepackage{hyperref}
\renewcommand\theequation{\thesection.\arabic{equation}}

\newcommand{\BA}{{\mathbb {A}}}

\newcommand{\BC}{{\mathbb {C}}}

\newcommand{\BQ}{{\mathbb {Q}}}
\newcommand{\BR}{{\mathbb {R}}}

\newcommand{\BZ}{{\mathbb {Z}}}

\newcommand{\CA}{{\mathcal {A}}}

\newcommand{\CE}{{\mathcal {E}}}
\newcommand{\CF}{{\mathcal {F}}}

\newcommand{\CH}{{\mathcal {H}}}

\newcommand{\CO}{{\mathcal {O}}}

\newcommand{\Fp}{{\mathfrak {p}}}

\newcommand{\RG}{{\mathrm {G}}}

\newcommand{\RO}{{\mathrm {O}}}

\newcommand{\RT}{{\mathrm {T}}}

\newcommand{\cusp}{{\mathrm{cusp}}}

\newcommand{\GL}{{\mathrm{GL}}}

\newcommand{\Ind}{{\mathrm{Ind}}}

\newcommand{\Mat}{{\mathrm{Mat}}}

\newcommand{\ns}{{\mathrm{ns}}}

\newcommand{\PGL}{{\mathrm{PGL}}}

\newcommand{\SO}{{\mathrm{SO}}}

\newcommand{\Sym}{{\mathrm{Sym}}}

\newcommand{\Sp}{{\mathrm{Sp}}}

\newcommand{\sm}{{\mathrm{sm}}}

\newcommand{\tr}{{\mathrm{tr}}}

\newcommand{\udl}{\underline}
\newcommand{\wt}{\widetilde}
\newcommand{\wh}{\widehat}

\newcommand{\ol}{\overline}
\newcommand{\ul}{\underline}

\newcommand{\bs}{\backslash}

\def\bks{{\backslash}}

\def\diag{{\rm diag}}

\newtheorem{thm}{Theorem}[section]

\newtheorem{lem}[thm]{Lemma}
\newtheorem{prop}[thm]{Proposition}
\newtheorem {conj}[thm]{Conjecture}

\newtheorem {ques/conj}[thm]{Question/Conjecture}

\newtheorem{defn}[thm]{Definition}
\newtheorem{rmk}[thm]{Remark}

\newtheorem{exmp}[thm]{Example}
\newtheorem{prob}[thm]{Problem}

\makeatletter

\newcommand{\Rmnum}[1]{\expandafter\@slowromancap\romannumeral #1@}
\makeatother

\begin{document}
\renewcommand{\theequation}{\arabic{equation}}
\numberwithin{equation}{section}

\title[Cuspidality of Global Arthur Packets]{On Cuspidality of Global Arthur Packets for Symplectic Groups}

\author{Dihua Jiang}
\address{School of Mathematics\\
University of Minnesota\\
Minneapolis, MN 55455, USA}
\email{dhjiang@math.umn.edu}

\author{Baiying Liu}
\address{School of Mathematics\\
Institute for Advanced Study\\
Einstein Drive\\
Princeton, New Jersey 08540 USA}
\email{liu@ias.edu}

\subjclass[2000]{Primary 11F70, 22E55; Secondary 11F30}

\date{\today}

\keywords{Arthur Parameters and Arthur Packets, Automorphic Discrete Spectrum of Classical Groups, Fourier Coefficients of Automorphic Forms}

\thanks{The research of the first named author is supported in part by the NSF Grants DMS--1301567, and that of the second
named author is supported in part by NSF Grants DMS--1302122, and in part by a postdoc research fund from Department of Mathematics, University of Utah}

\begin{abstract}
In \cite{Ar13}, Arthur classifies the automorphic discrete spectrum of symplectic groups up to global Arthur packets, based on the theory
of endoscopy. It is an interesting and basic question to ask: which global Arthur packets contain no cuspidal automorphic representations?
The investigation on this question can be regarded as a further development of the topics originated from the classical theory of singular
automorphic forms. The results obtained yield a better understanding of global Arthur packets and of the structure of
local unramified components of the cuspidal spectrum, and hence are closely related to the generalized Ramanujan problem as posted by Sarnak in
\cite{Sar05}.
\end{abstract}

\maketitle



\section{Introduction}

Let $F$ be a number field and $\BA$ be the ring of adeles of $F$. For an $F$-split classical group $\RG$, $\CA_2(\RG)$ denotes the set
of equivalence classes of all automorphic representations of $\RG(\BA)$ that occur in the discrete spectrum of the space of all
square-integrable automorphic forms on $\RG(\BA)$. The automorphic representations $\pi$ in the set $\CA_2(\RG)$ have been classified,
up to global Arthur packets, in the fundamental work of J. Arthur (\cite{Ar13}), based of the theory of endoscopy. More precisely,
for any $\pi\in\CA_2(\RG)$, there exists a global Arthur packet, denoted by $\wt{\Pi}_\psi(G)$, such that $\pi\in\wt{\Pi}_\psi(\RG)$ for
some global Arthur parameter $\psi\in\wt{\Psi}_2(\RG)$. Following \cite{Ar13},
a global Arthur parameter $\psi\in\wt{\Psi}_2(\RG)$ can be written formally as
\begin{equation}\label{ap}
\psi=(\tau_1,b_1)\boxplus(\tau_2,b_2)\boxplus\cdots\boxplus(\tau_r,b_r)
\end{equation}
where $\tau_j\in\CA_\cusp(\GL_{a_j})$ and $b_j\geq 1$ are integers. We refer to Section 2 for more details. A global Arthur parameter $\psi$ is
called {\sl generic}, following \cite{Ar13}, if the integers $b_j$ are one, i.e.
a generic global Arthur parameter $\psi$ can be written as
\begin{equation}\label{gap}
\psi=\phi=(\tau_1,1)\boxplus(\tau_2,1)\boxplus\cdots\boxplus(\tau_r,1).
\end{equation}

For a generic global Arthur parameter $\phi$ as in \eqref{gap}, the global Arthur packet $\wt{\Pi}_\phi(\RG)$ contains at least one member $\pi$
from the set $\CA_2(\RG)$. More precisely, this $\pi$ must belong to the subset $\CA_\cusp(\RG)$, i.e. it is cuspidal. This assertion
follows essential from the theory of automorphic descents of Ginzburg-Rallis-Soudry (\cite{GRS11}), as discussed in \cite{JL15a}. In fact,
as in \cite[Section 3.1]{JL15a}, one can show that a global Arthur parameter $\psi$ is generic if and only if the global Arthur packet
$\wt{\Pi}_\psi(\RG)$ contains a member $\pi\in\CA_\cusp(\RG)$ that has a nonzero Whittaker-Fourier coefficient (Theorem 3.4 in \cite{JL15a}).
It is not hard to show that when a global Arthur parameter $\psi=\phi$ is generic, the following holds:
$$
\wt{\Pi}_\phi(\RG)\cap\CA_2(\RG)\subset\CA_\cusp(\RG).
$$
All members in $\wt{\Pi}_\phi(\RG)\cap\CA_2(\RG)$ may be constructed via the {\sl twisted automorphic descents} as developed in \cite{JLXZ} and
more generally in \cite{JZ}.

In \cite{M08} and \cite{M11}, C. M\oe glin investigates the following problem: for a global Arthur parameter $\psi\in\wt{\Psi}_2(G)$, when does
the global Arthur packet $\wt{\Pi}_\psi(\RG)$ contain a non-cuspidal member in $\CA_2(\RG)$ and how to construct such non-cuspidal members if
exist? M\oe glin states her results in terms of her local  and global conjectures in the papers. We refer to \cite{M08} and \cite{M11}
for detailed discussions on those problems.

The objective of this paper is to investigate the following question: For a global Arthur parameter $\psi\in\wt{\Psi}_2(\RG)$, when does
the global Arthur packet $\wt{\Pi}_\psi(\RG)$ contain no cuspidal members, i.e. when is the intersection
$$
\wt{\Pi}_\psi(\RG)\cap\CA_\cusp(\RG)
$$
an empty set? The approach that we are taking to investigate this problem is based on our understanding of the structure of Fourier
coefficients of automoprhic forms associated to
nilpotent orbits or partitions, following the discussions and conjectures in \cite[Section 4]{J14} and \cite{JL15a}. This study can be regarded
as an extension of the fundamental work of R. Howe on the theory of singular automorphic forms using his notion of ranks for unitary representations
(\cite{H81}).

In this paper, we consider mainly the case that $G=\Sp_{2n}$, the symplectic groups. The method is applicable to other classical groups.
Due to technical reasons, we leave the discussion for other classical groups to our future work.

We start the discussion with a global Arthur parameter
$$
\psi=(\tau,2e)\boxplus(1,1)\in\wt{\Psi}_2(\Sp_{4e})
$$
with $\tau\in\CA_\cusp(\GL_2)$ of symplectic type. When $e=1$, the well-known example of Saito-Kurokawa provides irreducible cuspidal
automorphic representations in the global packet $\wt{\Pi}_\psi(\Sp_4)$, as constructed by Piatetski-Shapiro in \cite{PS83} using
global theta correspondences. This is the first known counter-example to the generalized Ramanujan conjecture,
which is not of unipotent cuspidal type.
Of course, the counter-examples of unipotent cuspidal type were constructed in 1979 by Howe and Piatetski-Shapiro in \cite{HPS79}, also using
global theta correspondences.
It was desirable to find such non-tempered cuspidal automorphic representations for general $\Sp_{2n}$ or even for general reductive groups.
In 1996, W. Duke and \"O. Imamoglu made a conjecture in \cite{DI96} that when $F=\BQ$, there exists the analogy of the Saito-Kurokawa type cuspidal
automorphic forms on $\Sp_{4e}$ for all integers $e\geq 1$. In terms of the endoscopic classification theory (\cite{Ar13}), the Duke-Imamoglu
conjecture asserts that when $F=\BQ$, the intersection
$$
\wt{\Pi}_\psi(\Sp_{4e})\cap\CA_\cusp(\Sp_{4e})
$$
is non-empty if the global Arthur parameter $\psi=(\tau,2e)\boxplus(1,1)$. This conjecture was confirmed positively by T. Ikeda in his
2001 Annals paper (\cite{Ik01}) and an extension to the case that $F$ is totally real in \cite{Ik}. The questions remain to ask:
\begin{enumerate}
\item What happens to the symplectic groups $\Sp_{4e+2}$?
\item What happens if $F$ is not totally real?
\end{enumerate}

For a general number field $F$, the authors joint with L. Zhang proved in \cite{JLZ13} that the intersection
$$
\wt{\Pi}_\psi(\Sp_{2n})\cap\CA_2(\Sp_{2n})
$$
is non-empty for a family of global Arthur parameters $\psi$, including the case that $\psi=(\tau,2e)\boxplus(1,1)$.
We explicitly constructed non-zero square-integrable residual representations in the global Arthur packets $\wt{\Pi}_\psi(\Sp_{2n})$
for a family of global Arthur parameters and hence confirmed the conjecture of M\oe glin in \cite{M08} and \cite{M11} for those cases.
Our main motivation in \cite{JLZ13} is to find {\sl automorphic kernel functions} for the automorphic integral transforms that
explicitly produce endoscopy correspondences as explained in \cite{J14}.

One of the main results in this paper confirms that when $F$ is {\sl totally imaginary} and $n\geq 5$, the intersection
$$
\wt{\Pi}_\psi(\Sp_{2n})\cap\CA_\cusp(\Sp_{2n})
$$
is empty for the global Arthur parameters $\psi=(\tau,2e)\boxplus(1,1)$ if $n=2e$ and $\psi=(\tau,2e+1)\boxplus(\omega_\tau,1)$ if $n=2e+1$,
where $\omega_\tau$ is the central character of $\tau$, and $\tau\in\CA_\cusp(\GL_2)$ is self-dual. Note that when $n=2e$, $\tau$ is of
symplectic type; and when $n=2e+1$, $\tau$ is of orthogonal type. This conclusion is a consequence of more general results obtained in Section 4,
where various versions of criteria for global Arthur packets containing no cuspidal members are given in
Theorems \ref{ncmain1}, \ref{ncmain2}, \ref{ncmain3}, and \ref{ncmain4}; and explicit examples are also discussed in Section 4.2.

On the other hand, we discuss the characterization of cuspidal automorphic representations with smallest possible Fourier coefficients, which are
called {\sl small} cuspidal representations in Section 2. We first explain how to re-interpret the result of Li that cuspidal automorphic representations of classical groups are non-singular, in terms of the Fourier coefficients associated to partitions or nilpotent orbits. This leads
to a question about the smallest possible Fourier coefficients for the cuspidal spectrum of classical groups, which is closely related to the
generalized Ramanujan problem as posted by P. Sarnak in 2005 (\cite{Sar05}). As a consequence of the discussion in Section 3, we find simple
criterion for $\Sp_{4n}$ that determines families of global Arthur parameters of unipotent type, with which the global Arthur packets
contains no cuspidal members (Theorem \ref{unip}). Examples and the relation of Theorem \ref{unip} with the work of S. Kudla and S. Rallis
(\cite{KR94}) are also discussed briefly in Section 3.

Generally speaking, by the endoscopic classification of the discrete spectrum of Arthur (\cite{Ar13}), the global Arthur parameters provide
the bounds for the Hecke eigenvalues or the exponents of the Satake parameters at the unramified local places for automorphic representations occurring in the
discrete spectrum. Since it is not clear how to deduce directly from the endoscopic classification which global Arthur packets contains no
cuspidal members, we apply the method of Fourier coefficients associated to unipotent orbits. Hence it is expected that our discussion improves
those bounds for the exponents of the Satake parameters of cuspidal spectrum if we find more global Arthur packets containing no cuspidal members. In Section 5, we
obtain a preliminary result towards the generalized Ramanujan problem. For general number fields, we show in Proposition 5.1 that
when $n=2e$ is even, the cuspidal automorphic representations of $\Sp_{4e}$ constructed by Piatetski-Shapiro and Rallis (\cite{PSR88})
achieve the worst bound, which is $\frac{n}{2}=e$, for the exponents of the Satake parameters of the cuspidal spectrum.
While in Proposition 5.2, we assume that $F$ is totally imaginary
and $n=2e+1\geq 5$ is odd, $\frac{n-1}{2}=e$ is an upper bound for the exponents of the Satake parameters of the cuspidal spectrum.
It needs more work to understand if the bound $\frac{n-1}{2}=e$ is {\sl sharp} when $F$ is totally imaginary
and $n=2e+1\geq 5$ is odd. It is also not clear that how to construct cuspidal representations with the worst bound for the exponents of the Satake parameters. We will come back to those issues in our future work.

In the last section (Section 6), we characterize the small cuspidal automorphic representations of $\Sp_{2n}(\BA)$ by means of Fourier coefficients
of Fourier-Jacobi type, and by the notion of hyper-cuspidal automorphic representations in the sense of Piatetski-Shapiro (\cite{PS83}). As
consequence, we prove (Theorem \ref{nohc}) that when $F$ is totally imaginary and $n\geq 5$, there does not exist any hyper-cuspidal automorphic
representation of $\Sp_{2n}(\BA)$.

The basic facts on the endoscopic classification of the discrete spectrum and the basic conjecture on the relations between the 
Fourier coefficients of automorphic forms and their global Arthur parameters are recalled in Section 2. Here we also recall the recent, 
relevant results of the authors, which are used in the rest of this paper.

Finally, we would like to thank J. Arthur, L. Clozel, J. Cogdell, R. Howe, R. Langlands, C. M{\oe}glin, P. Sarnak, F. Shahidi, R. Taylor, D. Vogan, and J.-L. Waldspurger for their interest in the problems discussed in
this paper and for their encouragement.

\subsection*{Acknowledgements} This material is based upon work supported by the National Science Foundation under agreement No. DMS-1128155. Any opinions, findings and conclusions or recommendations expressed in this material are those of the authors and do not necessarily reflect the views of the National Science Foundation.

\section{Fourier Coefficients and Global Arthur Packets}\label{secfcgapk}

\subsection{Fourier coefficients attached to nilpotent orbits}
In this section, we recall Fourier coefficients of automorphic forms attached to nilpotent orbits, following the formulation in \cite{GGS15},
which is slightly more general and easier to use than the one taken in \cite{J14} and \cite{JL15a}.
Let $\RG$ be a reductive group defined over $F$, or a central extension of finite degree.  Fix a nontrivial additive character $\psi$ of $F \bs \BA$.
Let $\frak{g}$ be the Lie algebra of $\RG(F)$ and $f$ be a nilpotent element in $\frak{g}$.
The element $f$ defines a function on $\frak{g}$:
\[
\psi_f: \frak{g} \rightarrow \BC^{\times}
\]
by $\psi_f(x) = \psi(\kappa(f,x))$, where $\kappa$ is the killing form on $\frak{g}$.

Given any semi-simple element $h \in \frak{g}$, under the adjoint action, $\frak{g}$ is decomposed to a direct sum of eigenspaces $\frak{g}^h_i$ of $h$ corresponding to eigenvalues $i$.
For any rational number $r \in \BQ$, let $\frak{g}^h_{\geq r} = \oplus_{r' \geq r} \frak{g}^h_{r'}$. The element
$h$ is called {\it rational semi-simple} if all its eigenvalues are in $\BQ$.
Given a nilpotent element $f$, a {\it Whittaker pair} is a pair $(h,f)$ with $h \in \frak{g}$ being a rational semi-simple element, and $f \in \frak{g}^h_{-2}$. The element $h$ in a Whittaker pair $(h,f)$ is called a {\it neutral element} for $f$ if $f \in \frak{g}^h_{-2}$ and the map $\frak{g}^h_0 \rightarrow \frak{g}^h_{-2}$ via $X \mapsto [X,f]$ is surjective.
For any nilpotent element $f \in \frak{g}$, by the Jacobson-Morozov Theorem, there is an $\frak{sl}_2$-triple $(e,h,f)$ such that $[h,f]=-2f$.
 In this case, $h$ is a neutral element for $f$.
By \cite[Lemma 2.2.1]{GGS15}, a Whittaker pair $(h,f)$ comes from an $\frak{sl}_2$-triple $(e,h,f)$ if and only if $h$ is a neutral element for $f$.
For any $X \in \frak{g}$, let $\frak{g}_X$ be the centralizer of $X$ in $\frak{g}$.

Given any Whittaker pair $(h,f)$, define an anti-symmetric form $\omega_f$ on $\frak{g}$ by $\omega_f(X,Y):=\kappa(f,[X,Y])$, where $\kappa$ is the killing form. We denote by $\omega = \omega_f$ when there is no confusion. Let $\frak{u}_h= \frak{g}^h_{\geq 1}$ and let $\frak{n}_h=\ker(\omega)$ be the radial of $\omega |_{\frak{u}_h}$. Then $[\frak{u}_h, \frak{u}_h] \subset \frak{g}^h_{\geq 2} \subset \frak{n}_h$. By \cite[Lemma 3.2.6]{GGS15}, $\frak{n}_h = \frak{g}^h_{\geq 2} + \frak{g}^h_1 \cap \frak{g}_f$.
Note that if the Whittaker pair $(h,f)$ comes from an $\frak{sl}_2$-triple $(e,h,f)$, then $\frak{n}_h=\frak{g}^h_{\geq 2}$. Let $U_{h}=\exp(\frak{u}_h)$ and $N_h=\exp(\frak{n}_h)$ be the corresponding unipotent subgroups of $\RG$. Define a character of $N_h$ by $\psi_f(n)=\psi(\kappa(f,\log(n)))$. Let $N_h' = N_h \cap \ker (\psi_f)$. Then $U_h/N_h'$ is a Heisenberg group with center $N_h/N_h'$.
It follows that for each Whittaker pair $(h,f)$, $\psi_f$ defines a character of $N_h(\BA)$ which is trivial on $N_h(F)$.

Assume that $\pi$ be an automorphic representation of $\RG(\BA)$. Define a
{\it degenerate Whittaker-Fourier coefficient} of $\varphi \in \pi$ by
\begin{equation}\label{fc}
\CF_{h,f}(\varphi)(g) = \int_{N_{h}(F) \bs N_{h}(\BA)} \varphi(ng)\ol{\psi}_f(n)dn, g \in \RG(\BA).
\end{equation}
Let $\CF_{h,f}(\pi)=\{\CF_{h,f}(\varphi) | \varphi\in \pi\}$.
If furthermore, $h$ is a neutral element for $f$, then $\CF_{h,f}(\varphi)$ is also called a {\it generalized Whittaker-Fourier coefficient} of $\varphi$.
The (global) {\it wave-front set} $\frak{n}(\pi)$ of $\pi$ is defined to the set of nilpotent orbits $\CO$ such that $\CF_{h,f}(\pi)$ is nonzero, for some Whittaker pair $(h,f)$ with $f \in \CO$ and $h$ being a neutral element for $f$. Note that if $\CF_{h,f}(\pi)$ is nonzero for some Whittaker pair $(h,f)$ with $f \in \CO$ and $h$ being a neutral element for $f$, then it is nonzero for any such Whittaker pair $(h,f)$, since the non-vanishing property of such Fourier coefficients does not depends on the choices of representatives of $\CO$. 
Let $\frak{n}^m(\pi)$ be the set of maximal elements in $\frak{n}(\pi)$ under the natural order of nilpotent orbits.
We recall \cite[Theorem C]{GGS15} as follows.

\begin{thm}[Theorem C, \cite{GGS15}]\label{ggsglobal}
Let $\pi$ be an automorphic representation of $\RG(\BA)$.
Given two Whittaker pairs $(h,f)$ and $(h',f')$, with $h$ being a neutral element for $f$, if $f \in \ol{\RG_{h'}(F) f'}$, where $\RG_{h'}$ is the centralizer of $h'$ in $\RG$, and $\CF_{h',f'}(\pi)$ is nonzero, then $\CF_{h,f}(\pi)$ is nonzero.
\end{thm}

Note that a particular case of Theorem \ref{ggsglobal} is that
$f=f'$. In this case, the condition $f \in \ol{\RG_{h'}(F) f'}$ is automatically satisfied, and hence Theorem \ref{ggsglobal} asserts in this case
that if $\CF_{h',f}(\pi)$ is nonzero, for some Whittaker pair $(h',f)$, then $\CF_{h,f}(\pi)$ is nonzero, for any Whittaker pair $(h,f)$ with $h$ being a neutral element for $f$.

When $\RG$ is a quasi-split classical group, it is known that the nilpotent orbits are parametrized by pairs $(\ul{p}, \ul{q})$, where $\ul{p}$ is a partition and $\ul{q}$ is a set of non-degenerate quadratic forms (see \cite{W01}).
When $\RG = \Sp_{2n}$, then $\ul{p}$ is symplectic partition, namely, odd parts occur with even multiplicities. When $\RG= \SO^{\alpha}_{2n}, \SO_{2n+1}$, then $\ul{p}$ is orthogonal partition, namely, even parts occur with even multiplicities. In these cases, let $\frak{p}^m(\pi)$ be the partitions corresponding to nilpotent orbits in $\frak{n}^m(\pi)$, that is,
the maximal nilpotent orbits in the wave-front set $\frak{n}(\pi)$ of the automorphic representation $\pi$.

\textbf{Convention}. {\it When $\RG$ is a quasi-split classical group, $\pi$ is an automorphic representation of $\RG(\BA)$, for any symplectic/orthogonal partition $\ul{p}$, by a Fourier coefficient attached to $\ul{p}$, we mean a generalized Whittaker-Fourier coefficient attached to an orbit $\CO$ parametrized by a pair $(\ul{p}, \ul{q})$ for some $\ul{q}$, that is, $\CF_{h,f}(\varphi)$, where $\varphi \in \pi$, $u\in \CO$ and $h$ is a neutral element for $f$. Sometimes, for convenience, we also write a Fourier coefficient attached to $\ul{p}$ as $\CF^{\psi_{\udl{p}}}(\varphi)$ without specifying the $F$-rational orbit $\CO$ and Whittaker pairs.}

Next, we recall the following result of \cite{JL15}, which is one of the main ingredients of this paper.

\begin{thm}[Theorem 5.3, \cite{JL15}]\label{ti}
Let $F$ be a totally imaginary number field.
And let $\pi$ be a cuspidal automorphic representation of $\Sp_{2n}(\BA)$ or $\wt{\Sp}_{2n}(\BA)$. Then there exists an even partition
(that is, consists of only even parts) in $\frak{p}^m(\pi)$, constructed in \cite{GRS03}, of the form
$$\ul{p}_{\pi}:=[(2n_1)^{s_1}(2n_2)^{s_2} \cdots (2n_r)^{s_r}],$$ with $2n_1 > 2n_2 > \cdots > 2n_r$ and $s_i \leq 4$ holds for $1 \leq i \leq r$.
\end{thm}

In this paper, we will consider two orders of partitions as follows. Given a partition $\ul{p}=[p_1p_2\cdots p_r]$, let $\lvert \ul{p} \rvert=\sum_{i=1}^r p_i$.

\begin{defn}\label{orders}
$(1)$. {\bf Lexicographical order}. Given two partitions $\ul{p}=[p_1p_2 \cdots p_r]$ with $p_1 \geq p_2 \geq \cdots \geq p_r$, and $\ul{q}=[q_1q_2 \cdots q_r]$ with $q_1 \geq q_2 \geq \cdots \geq q_r$, (add zeros at the end if needed) which may not be partitions of the same positive integer, i.e., $\lvert \ul{p} \rvert$ and $\lvert \ul{q} \rvert$ may not be equal. If there exists $1 \leq i \leq r$ such that $p_j = q_j$ for $1 \leq j \leq i-1$, and $p_i < q_i$, then we say that $\ul{p} < \ul{q}$ under the lexicographical order of partitions. Lexicographical order is a total order.

$(2)$. {\bf Dominance order}. Given two partitions $\ul{p}=[p_1p_2 \cdots p_r]$ with $p_1 \geq p_2 \geq \cdots \geq p_r$, and $\ul{q}=[q_1q_2 \cdots q_r]$ with $q_1 \geq q_2 \geq \cdots \geq q_r$ (add zeros at the end if needed), which again may not be partitions of the same positive integer, i.e., $\lvert \ul{p} \rvert$ and $\lvert \ul{q} \rvert$ may not be equal. If for any $1 \leq i \leq r$, $\sum_{j=1} p_j \leq \sum_{j=1}^i q_j$, then we say that $\ul{p} \leq \ul{q}$ under the dominance order of partitions. Dominance order is a partial order.
\end{defn}

\begin{rmk}
Given two partitions $\ul{p}$ and $\ul{q}$, if we do not specify which order of partitions, by $\ul{p} \leq \ul{q}$, we mean that it is under the dominance order of partitions.
\end{rmk}

\subsection{Automorphic discrete spectrum and Fourier coefficients}

In this paper, we consider mainly the symplectic groups. Although the methods are expected to work for all quasi-split classical groups, due to
the state of art in the current development of the theory, one knows much less when the classical groups are not of symplectic type.
Hence we will be concentrated on symplectic groups here and leave the discussion for other classical groups in future.

For symplectic group $\Sp_{2n}$, the endoscopic classification of the discrete spectrum was obtained by Arthur in \cite{Ar13}.
A preliminary statement of the endoscopic classification is recalled below.

\begin{thm}[Arthur \cite{Ar13}]
For any $\pi\in\CA_2(\Sp_{2n})$, there exists a global Arthur parameter
$$
\psi = \psi_1 \boxplus \cdots \boxplus \psi_r,
$$
such that $\pi\in\wt{\Pi}_\psi(\Sp_{2n})$, the global Arthur packet associated to $\psi$.
\end{thm}

The notation used in this theorem can be explained as follows.
Each $\psi_i=(\tau_i, b_i)$ is called a {\sl simple Arthur parameter}, where $\tau_i$ is an irreducible self-dual unitary cuspidal automorphic representation of $\GL_{a_i}(\BA)$ with central character $\omega_{\tau_i}$, $b_i \in \BZ_{\geq 1}$.
Every simple Arthur parameter $\psi_i$ is of orthogonal type. This means that if $\tau_i$ is of symplectic type,
that is, $L(s, \tau_i, \wedge^2)$ has a pole at $s=1$, then $b_i$ must be even; and if $\tau_i$ is of orthogonal type,
that is, $L(s, \tau_i, \Sym^2)$ has a pole at $s=1$, then $b_i$ must be odd. In order for the formal sum
$\psi = \psi_1 \boxplus \cdots \boxplus \psi_r$
to be a global Arthur parameter in $\wt{\Psi}_2(\Sp_{2n})$, one requires that
$2n+1 = \sum_{i=1}^r a_ib_i$, $\prod_{i=1}^r \omega_{\tau_i}^{b_i}=1$, and the simple parameters $\psi_i$'s are pair-wise different.

A global Arthur parameter $\psi$ is called {\sl generic}, following \cite{Ar13}, if the integers $b_i$ are one. The set of generic global
Arthur parameters is denoted by $\wt{\Phi}_2(\Sp_{2n})$.
A generic global Arthur parameter $\phi$ can be written as
$\phi=(\tau_1,1)\boxplus(\tau_2,1)\boxplus\cdots\boxplus(\tau_r,1)$.

\begin{conj}[Shahidi]\label{Shahidi}
For any generic global Arthur parameter $\phi=\boxplus_{i=1}^r(\tau_i,1) \in \wt{\Phi}_2(\Sp_{2n})$, there is an irreducible generic cuspidal automorphic representation $\pi$ of $\Sp_{2n}(\BA)$ belonging to $\wt{\Pi}_{\psi}(\Sp_{2n})$, and hence $\frak{p}^m(\pi)=\{[(2n)]\}$.
\end{conj}

This conjecture has been proved in \cite[Theorem 3.3]{JL15a}，using the automorphic descent of Ginzburg, Rallis and Soudry.
Note that by analyzing constant terms of residual representations, M{\oe}glin (\cite[Proposition 1.2.1]{M08}) shows that if there is a residual representation occurring in $\wt{\Pi}_{\psi}(\Sp_{2n})$, then the Arthur parameter is never generic. Hence we have
$$
\wt{\Pi}_\phi(\Sp_{2n})\cap\CA_2(\Sp_{2n})\subset\CA_\cusp(\Sp_{2n})
$$
for all generic global Arthur parameters $\phi\in\wt{\Phi}_2(\Sp_{2n})$.
For general Arthur parameters $\psi=\boxplus_{i=1}^r(\tau_i,b_i)\in\wt{\Psi}_2(\Sp_{2n})$, the following conjecture made in \cite{J14}
extends Conjecture \ref{Shahidi} naturally.

\begin{conj}[\cite{J14}]\label{J14}
Let $\RG$ be a quasi-split classical group. For a given global Arthur parameter $\psi=\boxplus_{i=1}^r(\tau_i,b_i)\in\wt{\Psi}_2(\Sp_{2n})$,
the partition $\eta(\ul{p}_{\psi})$, which is the Barbasch-Vogan dual of the partition $\ul{p}_{\psi}=[b_1^{a_1}b_2^{a_2} \cdots b_r^{a_r}]$
associated to the parameter $\psi$, has the following properties:
\begin{enumerate}
\item $\eta(\ul{p}_{\psi})$ is bigger than or equal to any $\ul{p} \in \frak{p}^m(\pi)$ for all
$\pi \in \wt{\Pi}_{\psi}(\Sp_{2n})\cap\CA_2(\Sp_{2n})$,
under the dominance order of partitions as in Definition \ref{orders}; and
\item there exists a $\pi \in \wt{\Pi}_{\psi}(\Sp_{2n})\cap\CA_2(\Sp_{2n})$ with $\eta(\ul{p}_{\psi}) \in \frak{p}^m(\pi)$.
\end{enumerate}
\end{conj}

\begin{rmk}\label{rmkbv}
The Barbasch-Vogan duality can be explained as follows.
Given a partition $\ul{p}=[p_1p_2 \cdots p_r]$ of $2n+1$, with $p_1 \geq p_2 \geq \cdots \geq p_r$,
by \cite[Definition A1]{BV85} and \cite[Section 3.5]{Ac03}, the Barbasch-Vogan dual $\eta(\ul{p})$ is defined to be $((\ul{p}^{-})_{\Sp})^t$.
More precisely, one has that
$\ul{p}^{-}=[p_1p_2\cdots (p_r-1)]$ and $(\ul{p}^{-})_{\Sp}$ is the biggest symplectic partition that is smaller than or equal to $\ul{p}^{-}$.
We refer to \cite[Lemma 6.3.8]{CM93} for the recipe of obtaining $(\ul{p}^{-})_{\Sp}$ from $\ul{p}^{-}$.  $(\ul{p}^{-})_{\Sp}$ is called
the symplectic collapse of $\ul{p}^{-}$. Finally, $((\ul{p}^{-})_{\Sp})^t$ is the transpose of $(\ul{p}^{-})_{\Sp}$.
By \cite[Lemma 3.3]{Ac03}, one has that $\eta(\ul{p})=((\ul{p}^t)^-)_{\Sp}$.
\end{rmk}

We recall the following result from \cite{JL15b}, which is also a main ingredient of this paper.

\begin{thm}[Theorem 1.3 and Proposition 6.4, \cite{JL15b}]\label{ub}
For a given global Arthur parameter $\psi=\boxplus_{i=1}^r(\tau_i,b_i) \in \widetilde{\Psi}_2(\Sp_{2n})$,
the Barbasch-Vogan dual $\eta(\ul{p}_{\psi})$ is bigger than or equal to any $\ul{p} \in \frak{p}^m(\pi)$
for every $\pi \in \wt{\Pi}_{\psi}(\Sp_{2n})\cap \CA_2(\Sp_{2n})$, under the lexicographical order of partitions as in Definition \ref{orders}.
\end{thm}

It is clear that when the global Arthur parameter $\psi=\phi$ is generic, the partition $\underline{p}_\phi=[1^{2n+1}]$, and
hence the partition $\eta(\underline{p}_\phi)=[(2n)]$,
which corresponds to the regular nilpotent orbit in $\frak{sp}_{2n}$.
Since any symplectic partition is smaller than or equal to $[(2n)]$, it follows that Conjecture \ref{J14} holds for all generic Arthur
parameters $\phi\in\wt{\Phi}_2(\Sp_{2n})$. Hence, it is more delicate to understand the lower bound for partitions
$\ul{p}\in\Fp^m(\pi)$ for all $\pi\in\CA_\cusp(\Sp_{2n})$. It is even harder to understand the lower bound for partitions
$\ul{p}\in\Fp^m(\pi)$ when $\pi\in\wt{\Pi}_{\psi}(\Sp_{2n})\cap\CA_\cusp(\Sp_{2n})$ for a given global Arthur parameter $\psi\in\wt{\Psi}_2(\Sp_{2n})$.

\begin{prob}\label{splb}
Find symplectic partitions $\ul{p}_0$ of $2n$ with the property that
\begin{enumerate}
\item there exists a $\pi\in\CA_\cusp(\Sp_{2n})$ such that $\ul{p}_0\in\Fp^m(\pi)$, but
\item for any $\pi\in\CA_\cusp(\Sp_{2n})$,  there does not exist a partition  $\ul{p}\in\Fp^m(\pi)$ such that $\ul{p}<\ul{p}_0$, under the dominance order of partitions as in Definition \ref{orders}.
\end{enumerate}
\end{prob}

This problem was motivated by the theory of singular automorphic representations of $\Sp_{2n}(\BA)$, which is briefly recalled
in the following section.

\subsection{Singular automorphic representations}\label{sar}
In this section, consider $\RG_n=\Sp_{2n}, \SO_{2n+1}$ or $\SO_{2n}$ to be split classical groups.
The theory of singular automorphic representations of $\RG_n(\BA)$ has been developed based on the notion of
ranks for unitary representations of Howe (\cite{H81}) and by the fundamental work of Li (\cite{Li92}).

When $\RG_n=\Sp_{2n}$ is the symplectic group, defined by the skew-symmetric matrix $J_n=\begin{pmatrix}0&w\\-w&0\end{pmatrix}$,
with $w=(w_{ij})_{n\times n}$ anti-diagonal.
Take $P_n=M_nU_n$ to be the Siegel parabolic subgroup of $\Sp_{2n}$. Hence $M_n\cong\GL_n$ and the elements of $U_n$ are of form
$$
u(X)=\begin{pmatrix}I_n&X\\ 0&I_n\end{pmatrix}.
$$
The Pontryagin duality tells us that the group of unitary characters
$U_n(\BA)$ which are trivial on $U_n(F)$ is isomorphic to $\Sym^2(F^n)$, i.e.
$$
\wh{U_n(F)\bks U_n(\BA)}\cong\Sym^2(F^n).
$$
The explicit isomorphism is given as follows. Take $\psi_F$ to be a nontrivial additive character of $F\bks\BA$. For any $T\in\Sym^2(F^n)$,
i.e. any $n\times n$ symmetric matrix $T$, the corresponding character $\psi_T$ is given by
$$
\psi_T(u(X)):=\psi_F(\tr(TwX)).
$$
The adjoint action of the Levi subgroup $\GL_n$ on $U_n$ induces an action of $\GL_n$ on $\Sym^2(F^n)$. For an automorphic form
$\varphi$ on $\Sp_{2n}(\BA)$, the $\psi_T$-Fourier coefficient is defined by
\begin{equation}\label{Tfc}
\CF^{\psi_T}(\varphi)(g)
:=
\int_{U_m(F)\bks U_n(\BA)}\varphi(u(X)g)\psi_T^{-1}(u(X))du(X).
\end{equation}
An automorphic form $\varphi$ on $\Sp_{2n}(\BA)$ is called {\sl non-singular} if $\varphi$ has a nonzero $\psi_T$-Fourier coefficient with
the $F$-rank of $T$ maximal, which is $n$. In other words, an automorphic form $\varphi$ on $\Sp_{2n}(\BA)$ is called
{\sl singular} if $\varphi$ has the property that if a $\psi_T$-Fourier coefficient $\CF^{\psi_T}(\varphi)$ is nonzero, then
$\det(T)=0$.

Based on his notion of ranks for unitary representations, Howe shows in \cite{H81} that if an automorphic form $\varphi$ on
$\Sp_{2n}(\BA)$ is singular, then $\varphi$ can be expressed as a linear combination of certain theta functions.
Li in \cite{Li89} shows that a cuspidal automorphic form of $\Sp_{2n}(\BA)$ with $n$ even is distinguished, i.e. $\varphi$ has nonzero $\psi_T$-Fourier coefficient with only one $\GL_n$-orbit of
non-degenerate $T$ if and only if $\varphi$ is in the space of theta lifting from the orthogonal group $\RO_T$ defined by $T$.
A family of explicit examples of such distinguished cuspidal automorphic representations of $\Sp_{2n}(\BA)$ with $n$ even was
constructed by Piatetski-Shapiro and Rallis in \cite{PSR88}.
Furthermore, Li proves in \cite{Li92} the following theorem.

\begin{thm}[\cite{Li92}]\label{li}
For any classical group $\RG_n$, cuspidal automorphic forms on $\RG_n(\BA)$ are non-singular.
\end{thm}

For orthogonal groups $\RG_n$, the singularity of automorphic forms can be defined as follows, following \cite{Li92}. Let $(V,q)$ be a non-degenerate
quadratic space defined over $F$ of dimension $m$ with the Witt index $n=[\frac{m}{2}]$. Let $X^+$ be a maximal totally isotropic subspace of $V$,
which has dimension $n$, and let $X^-$ be the maximal totally isotropic subspace of $V$ dual to $X^+$ with respect to $q$. Hence we have
the polar decomposition
$$
V=X^-+V_0+X^+
$$
with $V_0$ being the orthogonal complement of $X^-+X^+$, which has dimension less than or equal to one. The generalized flag
$$
\{0\}\subset X^+\subset V
$$
which defines a maximal parabolic subgroup $P_{X^+}$, whose Levi part $M_{X^+}$ is isomorphic to $\GL_{n}$ and whose unipotent radical
$N_{X^+}$ is abelian if $m$ is even; and is a two-step unipotent subgroup with its center $Z_{X^+}$ if $m$ is odd. When $m$ is even, we set $Z_{X^+}=N_{X^+}$.
Again, by the Pontryagin duality, we have
$$
\wh{Z_{X^+}(F)\bks Z_{X^+}(\BA)}\cong\wedge^2(F^{n}),
$$
which is given explicitly, as in the case $\Sp_{2n}$, by the following formula: For any $T\in\wedge^2(F^{[\frac{m}{2}]})$,
$$
\psi_T(z(X)):=\psi_F(\tr(TwX)).
$$
The adjoint action of the Levi subgroup $\GL_{n}$ on $Z_{X^+}$ induces an action of $\GL_{n}$ on
the space $\wedge^2(F^{n})$.
For an automorphic form $\varphi$ on $\RG(\BA)$, the $\psi_T$-Fourier coefficient is defined by
\begin{equation}\label{Tfc2}
\CF^{\psi_T}(\varphi)(g)
:=
\int_{Z_{X^+}(F)\bks Z_{X^+}(\BA)}\varphi(z(X)g)\psi_T^{-1}(z(X))dz(X).
\end{equation}
An automorphic form $\varphi$ on $\RG(\BA)$ is called {\sl non-singular} if $\varphi$ has a non-zero $\psi_T$-Fourier coefficient with
$T\in\wedge^2(F^{n})$ of maximal rank.

Following Section 2.1, we may reformulate the maximal rank Fourier coefficients of automorphic forms in terms of partitions, and
denote by $\udl{p}_\ns$ the partition corresponding to the non-singular Fourier coefficients. It is easy to figure out the following:
\begin{enumerate}\label{nspt}
\item When $\RG_n=\Sp_{2n}$, one has $\udl{p}_\ns=[2^n]$. This is a special partition for $\Sp_{2n}$.
\item When $\RG_n=\SO_{2n+1}$, one has
$$
\udl{p}_\ns
=
\begin{cases}
[2^{2e}1]&\ \text{if}\ n=2e;\\
[2^{2e}1^3]&\ \text{if}\ n=2e+1.
\end{cases}
$$
This is not a special partition of $\SO_{2n+1}$.
\item When $\RG_n=\SO_{2n}$, one has
$$
\udl{p}_\ns
=
\begin{cases}
[2^{2e}]&\ \text{if}\ n=2e;\\
[2^{2e}1^2]&\ \text{if}\ n=2e+1.
\end{cases}
$$
This is a special partition of $\SO_{2n}$.
\end{enumerate}
According to \cite{JLS15},  for any automorphic representation $\pi$, the set $\frak{p}^m(\pi)$ contains only special partitions.
Since the non-singular partition $\udl{p}_\ns$ is not special when $\RG_n=\SO_{2n+1}$,
the partitions contained in $\frak{p}^m(\pi)$ as $\pi$ runs in the cuspidal spectrum of $\RG_n$ should be bigger than or
equal to the following partition
$$
\udl{p}_\ns^{\RG_n}=
\begin{cases}
[32^{2e-2}1^2]&\ \text{if}\ n=2e;\\
[32^{2e-2}1^4]&\ \text{if}\ n=2e+1.
\end{cases}
$$
Following \cite{CM93}, $\udl{p}_\ns^{\RG_n}$ denotes the $\RG_n$-expansion of the partition $\udl{p}_\ns$, i.e., the smallest special partition which is bigger than or equal to $\udl{p}_\ns$.
Of course, when $\RG_n=\Sp_{2n}$ or $\SO_{2n}$, one has that $\udl{p}_\ns^{\RG_n}=\udl{p}_\ns$.

\begin{prop}\label{lbpt}
For split classical group $\RG_n$, the $\RG_n$-expansion of the non-singular partition, $\udl{p}_\ns^{\RG_n}$,
is a lower bound for partitions in the set $\Fp^m(\pi)$ as $\pi$ runs in the cuspidal spectrum of $\RG_n$.
\end{prop}

It is natural to ask whether the lower bound $\udl{p}_\ns^{\RG_n}$ is sharp. This is to construct or find an irreducible cuspidal
automorphic representation $\pi$ of $\RG_n(\BA)$ with the property that $\udl{p}_\ns^{\RG_n}\in\Fp^m(\pi)$.

When $\RG_n=\Sp_{4e}$ with $n=2e$ even, and when $F$ is totally real, the examples constructed by T. Ikeda (\cite{Ik01} and \cite{Ik})
are irreducible cuspidal
automorphic representations $\pi$ of $\Sp_{4e}(\BA)$ with the global Arthur parameter $\psi=(\tau,2e)\boxplus(1,1)$, where $\tau\in\CA_\cusp(\GL_2)$
is of symplectic type. By Theorem \ref{ub}, for any partition $\ul{p}\in\Fp^m(\pi)$, we should have
$$
\ul{p}\leq \eta(\udl{p}_\psi)=[2^{2e}]=\udl{p}_\ns^{\Sp_{4e}},
$$
under the lexicographical order of partitions, which automatically implies that $\ul{p}\leq[2^{2e}]=\udl{p}_\ns^{\Sp_{4e}}$ under the dominance order of partitions.
On the other hand, by Theorem \ref{li}, for any partition $\ul{p}\in\Fp^m(\pi)$, we must have
$$
\udl{p}_\ns^{\Sp_{4e}}=[2^{2e}]\leq\ul{p},
$$
under the dominance order of partitions.
It follows that $\Fp^m(\pi)=\{[2^{2e}]=\udl{p}_\ns^{\Sp_{4e}}\}$.

\begin{prop}\label{ik}
When $F$ is totally real, the non-singular partition $\udl{p}_\ns^{\Sp_{4e}}=\udl{p}_\ns=[2^{2e}]$ is the sharp lower bound in the sense that
for all $\pi\in\CA_\cusp(\Sp_{4e})$, the partition $\udl{p}_\ns^{\Sp_{4e}}\in\Fp(\pi)$ and
there exists a $\pi\in\CA_\cusp(\Sp_{4e})$, as constructed in \cite{Ik01} and \cite{Ik}, such that $\udl{p}_\ns^{\Sp_{4e}}\in\Fp^m(\pi)$.
\end{prop}

It is clear that the assumption that $F$ must be totally real is substantial in the construction of Ikeda in \cite{Ik01} and \cite{Ik}.
However, there is no known approach to proceed the similar construction when $F$ is not totally real. We are going to discuss the situation
in the following sections when $F$ is totally imaginary, which leads to a totally different conclusion.

Also the situation is different when we consider orthogonal groups.
For $\RG_n$ to $\SO_{2n+1}$ or $\SO_{2n}$, in spirit of a conjecture of Ginzburg (\cite{G06}), any partition $\ul{p}$ in
$\Fp(\pi)$ with $\pi\in\CA_\cusp(\RG_n)$  should contain only odd parts. Hence it is reasonable to conjecture the existence of a lower bound
which is better than the one determined by non-singularity of cuspidal automorphic representations.

\begin{conj}\label{lbso}
For $\RG_n$ to be $F$-split $\SO_{2n+1}$ or $\SO_{2n}$,
the sharp lower bound partition $\udl{p}_0^{\RG_n}$ for $\ul{p} \in \frak{p}(\pi)$, as $\pi$ runs in $\CA_\cusp(\RG_n)$, is given as follows:
\begin{enumerate}
\item When $\RG_n=\SO_{2n+1}$,
$$
\udl{p}_0^{\SO_{2n+1}}
=
\begin{cases}
[3^e1^{e+1}]&\ \text{if}\ n=2e;\\
[3^{e+1}1^e]&\ \text{if}\ n=2e+1.
\end{cases}
$$
\item When $\RG_n=\SO_{2n}$,
$$
\udl{p}_0^{\SO_{2n}}
=
\begin{cases}
[3^e1^e]&\ \text{if}\ n=2e;\\
[53^{e-1}1^e]&\ \text{if}\ n=2e+1.
\end{cases}
$$
\end{enumerate}
\end{conj}

We note that a shape lower bound partition for the Fourier coefficients of all irreducible cuspidal representations of $\RG_n(\BA)$ involves
deep arithmetic of the base field $F$, which is one of the main concerns in our investigation.
Following the line of ideas in \cite{H81} and \cite{Li92}, we define the following set of {\sl small partitions} for the cuspidal spectrum of $\RG_n(\BA)$:
\begin{equation}\label{smallpartitions}
\Fp_{\sm}^{\RG_n, F} := \min \{\ul{p} \in \Fp^m(\pi)\ |\ \text{ for some}\ \pi\in\CA_\cusp(\RG_n) \}.
\end{equation}
We call a $\pi\in\CA_\cusp(\RG_n)$ {\sl small} if $\frak{p}^m(\pi) \cap \frak{p}_{\sm}^{\RG_n, F}$ is not empty. Our discussion
for small cuspidal automorphic representations will resume in Section \ref{SCAR}.

\section{On Cuspidality for General Number Fields}\label{CGNF}

In this section, we assume that $F$ is a general number field. We mainly consider the {\sl cuspidality problem} for the global Arthur packets
with a family of global Arthur parameters of form:
$$
\psi=(\chi,b) \oplus (\tau_2, b_2) \oplus \cdots \oplus (\tau_r,b_r)\in\wt{\Psi}_2(\Sp_{2n}).
$$
When $b$ is large, it is most likely that the corresponding global Arthur packet $\wt{\Pi}_\psi(\Sp_{2n})$ contains no cuspidal members.

Recall from Section 2.2 that by Conjecture \ref{J14} for $\RG_n=\Sp_{2n}$, for any $\pi\in\wt{\Pi}_{\psi}(\Sp_{2n})\cap\CA_2(\Sp_{2n})$,
it is expected that for any partition $\ul{p}\in\Fp^m(\pi)$, one should have
\begin{equation}\label{upbfc2}
\ul{p}\leq \eta(\ul{p}_{\psi}),
\end{equation}
under the dominance order of partitions. We will take this as an {\sl assumption} for the discussion in this section.

For $\psi=(\chi,b) \oplus (\tau_2, b_2) \oplus \cdots \oplus (\tau_r,b_r) \in \widetilde{\Psi}_2(\Sp_{2n})$, with $\chi$ a quadratic character,
the partition associated to $\psi$ is
$$
\ul{p}_{\psi}=[(b)^1 (b_2)^{a_2} \cdots (b_r)^{a_r}].
$$
By the  definition of Arthur parameters for $\Sp_{2n}$, $b$ is automatically odd.
As explained in Remark \ref{rmkbv},
$\eta(\ul{p}_{\psi})
=((\ul{p}_{\psi}^t)^{-})_{\Sp}$.
Assume that $b >b_0:=\max(b_2, \ldots, b_r)$, then
$$\ul{p}_{\psi}^t=[(1)^b]+[(a_2)^{b_2}]+\cdots + [(a_r)^{b_r}]$$
has the form $[(1+\sum_{i=2}^r a_i)p_2 \cdots p_{b_0} (1)^{b-b_0}]$, and
$$(\ul{p}_{\psi}^t)^{-}=[(1+\sum_{i=2}^r a_i)p_2 \cdots p_{b_0} (1)^{b-b_0-1}].$$
After taking the symplectic collapse, $\eta(\ul{p}_{\psi})
= ((\ul{p}_{\psi}^t)^{-})_{\Sp}$ has the form
$$[q_1 q_2\cdots q_k (1)^m],$$
with $m \leq b-1-\sum_{i=2}^r b_i$, and $k+m=b-1$.

If there is a $\pi\in\wt{\Pi}_{\psi}(\Sp_{2n})\cap\CA_\cusp(\Sp_{2n})$, by Theorem \ref{li}, $\pi$ has a nonzero Fourier coefficient attached to the partition $[2^n]$. It is clear that $b>n+1$ if and only if  $[2^n]$ is either bigger than or not related to the above partition $[q_1 q_2\cdots q_k (1)^m]$. Hence, we have the following result.

\begin{thm}\label{unip}
Assume that $\eqref{upbfc2}$ holds. For
$$\psi=(\chi,b) \oplus (\tau_2, b_2) \oplus \cdots \oplus (\tau_r,b_r)\in \widetilde{\Psi}_2(\Sp_{2n})$$
with  $\chi$ a quadratic character, if $b>n+1$, then the intersection
$$\wt{\Pi}_{\psi}(\Sp_{2n})\cap\CA_\cusp(\Sp_{2n})$$
is empty.
\end{thm}

Here is an example illustrating the theorem.

\begin{exmp}
Consider $\psi = (\chi, 7) \oplus (\tau,2)\in\widetilde{\Psi}_2(\Sp_{10})$, where $\chi=1_{\GL_1(\BA)}$, and $\tau\in\CA_\cusp(\GL_2)$
with $L(s, \tau, \wedge^2)$ having a pole at $s=1$. $\ul{p}_{\psi}=[72^2]$ and $\eta(\ul{p}_{\psi})=[3^21^4]$, which is not related to $[2^5]$. Hence, by the assumption that $\eqref{upbfc2}$ holds, there is no cuspidal members in the global Arthur packet $\wt{\Pi}_{\psi}(\Sp_{10})$.
\end{exmp}

\begin{rmk}\label{rmkKR}
In \cite[Theorem 7.2.5]{KR94}, Kudla and Rallis show that for a given $\pi\in\CA_\cusp(\Sp_{2n})$ and a quadratic character $\chi$,
the $L$-function $L(s, \pi \times \chi)$ has the right-most possible pole at $s=1+[\frac{n}{2}]$. This implies that
the simple global Arthur parameter of type $(\chi,b)$ occurring in the global Arthur parameter of $\pi$ must have the condition that
$b$ is at most $1+2[\frac{n}{2}]$. Because $b$ has to be odd in this case, it follows that $b$ is at most $n+1$ if $n$ is even, and
$b$ is at most $n$ if $n$ is odd. In any case, one obtains that if $b > n+1$, then the simple global Arthur parameter of type $(\chi,b)$ can
not occur in the global Arthur parameter of $\pi$ for any $\pi\in\CA_\cusp(\Sp_{2n})$. This matches the result in the above theorem.
\end{rmk}

\section{On Cuspidality for Totally Imaginary Fields}\label{CTIF}

In this section, we assume that $F$ is a totally imaginary number field. We show that there are more global Arthur packets
contain no cuspidal members. It is an interesting question to discover the significance of such a difference depending on the
arithmetic of the ground field $F$.

\subsection{On criteria for cuspidality}
For any $\ul{a}=(a_1, a_2, \ldots, a_r) \in \BZ_{\geq 1}^r$, define a set $B_{\ul{a}}$, depending only on $\ul{a}$,
to be the subset of $\BZ_{\geq 1}^r$ that consists of all $r$-tuples $\ul{b}=(b_1, b_2, \ldots, b_r)$ with the property:
There are some self-dual $\tau_i\in\CA_\cusp(\GL_{a_i})$ for $1 \leq i \leq r$, such that
$$
\psi=(\tau_1, b_1) \boxplus (\tau_2, b_2) \boxplus \cdots \boxplus (\tau_r, b_r)
$$
belongs to $\wt{\Psi}_2(\Sp_{2n})$ for some $n\geq 1$ with $2n+1=\sum_{i=1}^ra_ib_i$.
We define an integer $N_{\ul{a}}$ depending only on $\ul{a}$ by
\begin{equation}\label{Na}
N_{\ul{a}}=\begin{cases}
(\sum_{i=1}^r a_i)^2+2(\sum_{i=1}^r a_i) &\textit{if}\  \sum_{i=1}^r a_i \textit{is even};\\
(\sum_{i=1}^r a_i)^2-1 & \textit{otherwise}.
\end{cases}
\end{equation}

\begin{thm}\label{ncmain1}
Assume that $F$ is a totally imaginary number field.
Given an $\ul{a}=(a_1, a_2, \ldots, a_r) \in \BZ_{\geq 1}^r$ that defines the set $B_{\ul{a}}$ and the integer $N_{\ul{a}}$ as above.
For any $\ul{b}=(b_1, b_2, \ldots, b_r) \in B_{\ul{a}}$, write $2n+1=\sum_{i=1}^ra_ib_i$.
If the condition
$$
2n=(\sum_{i=1}^r a_ib_i)-1 > N_{\ul{a}}
$$
holds, then for any global Arthur parameter $\psi$ of the form
$$
\psi=(\tau_1, b_1) \boxplus (\tau_2, b_2) \boxplus \cdots \boxplus (\tau_r, b_r)\in\wt{\Psi}_2(\Sp_{2n}),
$$
with $\tau_i\in\CA_\cusp(\GL_{a_i})$ for $i=1,2,\cdots,r$, $\wt{\Pi}_\psi(\Sp_{2n}) \cap \CA_2(\Sp_{2n})$ contains no cuspidal members.
\end{thm}

\begin{proof}
By assumption, $\psi=\boxplus_{i=1}^r (\tau_i, b_i)$ belongs to $\wt{\Psi}_2(\Sp_{2n})$.
Recall that
$$
\ul{p}_{\psi}=[(b_1)^{a_1} (b_2)^{a_2} \cdots (b_r)^{a_r}]
$$
is the partition of $2n+1$ attached to $\psi$.
By Remark \ref{rmkbv},
$\eta(\ul{p}_{\psi})=((\ul{p}_{\psi}^t)^{-})_{\Sp}$.
Then the Barbasch-Vogan dual $\eta(\ul{p}_{\psi})$ has the following form
\begin{equation}\label{bvsp}
[(\sum_{i=1}^r a_i)p_2 \cdots p_s]_{\Sp},
\end{equation}
where $\sum_{i=1}^r a_i \geq p_2 \geq \cdots \geq p_s$.
After taking the symplectic collapse of the partition in \eqref{bvsp}, one obtains that $\eta(\ul{p}_{\psi})$ must be one of the
following three possible forms:
\begin{enumerate}
\item It equals $[(\sum_{i=1}^r a_i)p_2 \cdots p_s]$ if $\sum_{i=1}^r a_i$ is even and
$$\sum_{i=1}^r a_i \geq p_2 \geq \cdots \geq p_s.$$
\item It equals $[(\sum_{i=1}^r a_i)p_2 \cdots p_s]$ if $\sum_{i=1}^r a_i$ is odd and
$$\sum_{i=1}^r a_i \geq p_2 \geq \cdots \geq p_s.$$
\item It equals $[((\sum_{i=1}^r a_i)-1)p_2 \cdots p_s]$ if $(\sum_{i=1}^r a_i)$ is odd and
$$(\sum_{i=1}^r a_i)-1 \geq p_2 \geq \cdots \geq p_s.$$
\end{enumerate}
Assume that $\pi$ belongs to $\wt{\Pi}_{\psi}(\Sp_{2n})\cap\CA_\cusp(\Sp_{2n})$. By Theorem \ref{ti}, one may assume that
$$
\ul{p}_{\pi} = [(2n_1)^{s_1} (2n_2)^{s_2} \cdots (2n_r)^{s_r}]\in\frak{p}^m(\pi)
$$
with $n_1 > n_2 > \cdots > n_k \geq 1$ and with the property that $1 \leq s_i \leq 4$ holds for $1 \leq i \leq r$.

{\bf Case 1:}\ By Theorem \ref{ub}, we have $2n_1 \leq \sum_{i=1}^r a_i$.
It follows that
\begin{eqnarray*}
2n
&=&\sum_{i=1}^r 2n_is_i \\
&\leq& 4(2 + 4 + 6 + \cdots + \sum_{i=1}^r a_i)\\
&=&(\sum_{i=1}^r a_i)^2+2(\sum_{i=1}^r a_i)=N_{\ul{a}}.
\end{eqnarray*}

{\bf Cases 2 and 3:}\ By Theorem \ref{ub}, we have $2n_1 \leq (\sum_{i=1}^r a_i)-1$.
It follows that
\begin{eqnarray*}
2n
&=&\sum_{i=1}^r 2n_is_i \\
&\leq& 4(2 + 4 + 6 + \cdots + (\sum_{i=1}^r a_i)-1)\\
&=&(\sum_{i=1}^r a_i)^2-1=N_{\ul{a}}.
\end{eqnarray*}

Now it is easy to check that for any $r$-tuple $\ul{b}=(b_1, b_2, \ldots, b_r) \in B_{\ul{a}}$,
if $2n=(\sum_{i=1}^r a_ib_i)-1 > N_{\ul{a}}$, then the global Arthur packets $\wt{\Pi}_\psi(\Sp_{2n})$ associated to
any global Arthur parameters of the form
$$
\psi=(\tau_1, b_1) \boxplus (\tau_2, b_2) \boxplus \cdots \boxplus (\tau_r, b_r)
$$
contain no cuspidal members. This completes the proof of the theorem.
\end{proof}

Note that in Theorem \ref{ncmain1}, for a given $\ul{a}=(a_1, a_2, \ldots, a_r) \in \BZ_{\geq 1}^r$, the integer $n$ defining
the group $\Sp_{2n}$ depends on the choice of $\ul{b}=(b_1, b_2, \ldots, b_r) \in B_{\ul{a}}$. We may reformulate the result for a given
group $\Sp_{2n}$ as follows.

For any $r$-tuple $\ul{a}=(a_1, a_2, \ldots, a_r) \in \BZ_{\geq 1}^r$, define $B_{\ul{a}}^{2n}$ to be the subset of $\BZ_{\geq 1}^r$,
consisting of $r$-tuples $\ul{b}=(b_1, b_2, \ldots, b_r)$ such that
$$
\psi=(\tau_1, b_1) \boxplus (\tau_2, b_2) \boxplus \cdots \boxplus (\tau_r, b_r)\in\wt{\Psi}_2(\Sp_{2n})
$$
for some self-dual $\tau_i\in\CA_\cusp(\GL_{a_i})$ with $1 \leq i \leq r$. Note that this set $B_{\ul{a}}^{2n}$ could be empty in this formulation.
The integer $N_{\ul{a}}$ is defined to be the same as in \eqref{Na}.
Theorem \ref{ncmain1} can be reformulated as follows.

\begin{thm}\label{ncmain2}
Assume that $F$ is a totally imaginary number field and
that $\ul{a}=(a_1, a_2, \ldots, a_r) \in \BZ_{\geq 1}^r$ has a non-empty $B_{\ul{a}}^{2n}$.
If $2n > N_{\ul{a}}$, then for any global Arthur parameter $\psi$ of the form
$$
\psi=(\tau_1, b_1) \boxplus (\tau_2, b_2) \boxplus \cdots \boxplus (\tau_r, b_r)\in\wt{\Psi}_2(\Sp_{2n}),
$$
with $\tau_i\in\CA_\cusp(\GL_{a_i})$ for $i=1,2,\cdots,r$, $\wt{\Pi}_\psi(\Sp_{2n}) \cap \CA_2(\Sp_{2n})$ contains no cuspidal members.
\end{thm}

On one hand, the integer $N_{\ul{a}}$ is not hard to calculate. This makes Theorems \ref{ncmain1} and \ref{ncmain2} easy to use.
On the other hand, the integer $N_{\ul{a}}$ depends only on $\ul{a}$, and hence may not carry enough information for some applications.
Next, we try to improve the above bound $N_{\ul{a}}$, by defining a new bound $N_{\ul{a},\ul{b}}^{(1)}$, depending on both $\ul{a}$ and $\ul{b}$.

For a partition $\ul{p}=[p_1p_2\cdots p_r]$, set $\lvert \ul{p} \rvert=\sum_{i=1}^r p_i$.
Given an $\ul{a}=(a_1, a_2, \ldots, a_r) \in \BZ_{\geq 1}^r$ that defines the set $B_{\ul{a}}$.
For any $\ul{b}=(b_1, b_2, \ldots, b_r)$ belonging to the set $B_{\ul{a}}$ that defines the integer $n$ with
$$
2n+1=\sum_{i=1}^ra_ib_i,
$$
the new bound $N_{\ul{a},\ul{b}}^{(1)}$ is defined to be maximal value of $\lvert \ul{p} \rvert$ for all symplectic partitions $\ul{p}$,
which may not be a partition of $2n$,
satisfying the following conditions:
\begin{enumerate}
\item $\ul{p} \leq \eta(\ul{p}_{\psi})$ under the {\sl lexicographical order} of partitions as in Definition \ref{orders}, and
\item $\ul{p}$ has the form $[(2n_1)^{s_1}(2n_2)^{s_2} \cdots (2n_r)^{s_r}]$ with $2n_1 > 2n_2 > \cdots > 2n_r$ and
$s_i \leq 4$ for $1 \leq i \leq r$.
\end{enumerate}
Note that the integer $N_{\ul{a},\ul{b}}^{(1)}$ depends on $\ul{b}$ through Condition (1) above. For this new bound,
we have the following result.

\begin{thm}\label{ncmain3}
Assume that $F$ is a totally imaginary number field. Given an
$\ul{a}=(a_1, a_2, \ldots, a_r) \in \BZ_{\geq 1}^r$ that defines the set $B_{\ul{a}}$.
For any $\ul{b}=(b_1, b_2, \ldots, b_r) \in B_{\ul{a}}$, if $2n=(\sum_{i=1}^r a_ib_i)-1 > N_{\ul{a},\ul{b}}^{(1)}$,
then for any global Arthur parameter $\psi$ of the form
$$
\psi=(\tau_1, b_1) \boxplus (\tau_2, b_2) \boxplus \cdots \boxplus (\tau_r, b_r)\in\wt{\Psi}_2(\Sp_{2n}),
$$
with $\tau_i\in\CA_\cusp(\GL_{a_i})$ for $i=1,2,\cdots,r$, $\wt{\Pi}_\psi(\Sp_{2n}) \cap \CA_2(\Sp_{2n})$ contains no cuspidal members.
\end{thm}

\begin{proof}
Assume that there is a $\pi \in \wt{\Pi}_\psi(\Sp_{2n})\cap\CA_\cusp(\Sp_{2n})$.  By Theorem \ref{ub},
for any $\ul{p} \in \frak{p}^m(\pi)$, which is a partition of $2n$, we must have that $\ul{p} \leq \eta(\ul{p}_{\psi})$
under the lexicographical order of partitions. In particular, the even partition $\ul{p}_{\pi} \in \frak{p}^m(\pi)$,
constructed in \cite{GRS03}, enjoys this property. On the other hand, since $F$ is totally imaginary, by Theorem \ref{ti},
$\ul{p}_{\pi}$ has the form $[(2n_1)^{s_1}(2n_2)^{s_2} \cdots (2n_r)^{s_r}]$ with $2n_1 > 2n_2 > \cdots > 2n_r$ and
$s_i \leq 4$ for $1 \leq i \leq r$. Hence, $\ul{p}_{\pi}$ satisfies the above two conditions defining the bound $N_{\ul{a},\ul{b}}^{(1)}$.
It follows that $N_{\ul{a},\ul{b}}^{(1)} \geq 2n=\lvert \ul{p}_\pi \rvert$. This contradicts the assumption that $2n > N_{\ul{a},\ul{b}}^{(1)}$.
\end{proof}

If we assume that Part (1) of Conjecture \ref{J14} holds, namely, $\eta(\ul{p}_{\psi})$ is bigger than or equal to any $\ul{p} \in \frak{p}^m(\pi)$, under the dominance order of partitions, for all $\pi \in \wt{\Pi}_{\psi}(\Sp_{2n})\cap\CA_2(\Sp_{2n})$,
we may replace the bound $N_{\ul{a},\ul{b}}^{(1)}$ by an even better bound $N_{\ul{a},\ul{b}}^{(2)}$ as follows.

Given an $\ul{a}=(a_1, a_2, \ldots, a_r) \in \BZ_{\geq 1}^r$ that defines the set $B_{\ul{a}}$.
For any $\ul{b}=(b_1, b_2, \ldots, b_r) \in B_{\ul{a}}$ that defines the integer $n$ with
$$
2n+1=\sum_{i=1}^ra_ib_i,
$$
the new bound $N_{\ul{a},\ul{b}}^{(2)}$ is defined to be the maximal value of $\lvert \ul{p} \rvert$ for all symplectic partitions $\ul{p}$,
which may not be a partition of $2n$, satisfying the following conditions:
  \begin{enumerate}
  \item $\ul{p} \leq \eta(\ul{p}_{\psi})$ under the {\sl dominance order} of partitions, as in Definition \ref{orders}, and
  \item $\ul{p}$ has the form $[(2n_1)^{s_1}(2n_2)^{s_2} \cdots (2n_r)^{s_r}]$ with $2n_1 > 2n_2 > \cdots > 2n_r$ and $s_i \leq 4$ holds for $1 \leq i \leq r$.
  \end{enumerate}
It is clear that the integer $N_{\ul{a},\ul{b}}^{(2)}$ depends on $\ul{b}$ through Condition (1) above.
By assuming Part (1) of Conjecture \ref{J14}, we can prove the following with this new bound.

\begin{thm}\label{ncmain4}
Assume that $F$ is a totally imaginary number field, and that Part (1) of Conjecture \ref{J14} is true.
Given an $\ul{a}=(a_1, a_2, \ldots, a_r) \in \BZ_{\geq 1}^r$ that defines the set $B_{\ul{a}}$.
For any $\ul{b}=(b_1, b_2, \ldots, b_r) \in B_{\ul{a}}$, if $2n=(\sum_{i=1}^r a_ib_i)-1 > N_{\ul{a},\ul{b}}^{(2)}$, then
for any global Arthur parameter $\psi$ of the form
$$
\psi=(\tau_1, b_1) \boxplus (\tau_2, b_2) \boxplus \cdots \boxplus (\tau_r, b_r)\in\wt{\Psi}_2(\Sp_{2n}),
$$
with $\tau_i\in\CA_\cusp(\GL_{a_i})$ for $i=1,2,\cdots,r$, $\wt{\Pi}_\psi(\Sp_{2n}) \cap \CA_2(\Sp_{2n})$ contains no cuspidal members.
\end{thm}

\begin{proof}
Assume that there is a $\pi \in \wt{\Pi}_\psi(\Sp_{2n})\cap\CA_\cusp(\Sp_{2n})$. By Part (1) of Conjecture \ref{J14},
for any $\ul{p} \in \frak{p}^m(\pi)$, which is a partition of $2n$, we must have that $\ul{p} \leq \eta(\ul{p}_{\psi})$
under the dominance order of partitions. In particular, the even partition $\ul{p}_{\pi} \in \frak{p}^m(\pi)$, constructed in \cite{GRS03},
enjoys this property. On the other hand, since $F$ is totally imaginary, by Theorem \ref{ti}, $\ul{p}_{\pi}$ has
the form $[(2n_1)^{s_1}(2n_2)^{s_2} \cdots (2n_r)^{s_r}]$ with $2n_1 > 2n_2 > \cdots > 2n_r$ and
$s_i \leq 4$ for $1 \leq i \leq r$. Hence, $\ul{p}_{\pi}$ satisfies the above two conditions defining the bound $N_{\ul{a},\ul{b}}^{(2)}$.
It follows that $N_{\ul{a},\ul{b}}^{(2)} \geq 2n=\lvert \ul{p}_\pi \rvert$. This contradicts the assumption that $2n > N_{\ul{a},\ul{b}}^{(2)}$.
\end{proof}

First, it is clear that $N_{\ul{a}} \geq N_{\ul{a},\ul{b}}^{(1)} \geq N_{\ul{a},\ul{b}}^{(2)}$. We expect that the bound $N_{\ul{a},\ul{b}}^{(2)}$ is sharp. Namely, for any $\ul{b}=(b_1, b_2, \ldots, b_r) \in B_{\ul{a}}$ with $\sum_{i=1}^r a_ib_i = N_{\ul{a},\ul{b}}^{(2)}+1$, we expect that
any global packet $\wt{\Pi}_\psi(\Sp_{N_{\ul{a},\ul{b}}^{(2)}})$ associated to any global Arthur parameter $\psi$ of the form
$$
\psi=(\tau_1, b_1) \boxplus (\tau_2, b_2) \boxplus \cdots \boxplus (\tau_r, b_r)\in\wt{\Psi}_2(\Sp_{N_{\ul{a},\ul{b}}^{(2)}}),
$$
with $\tau_i\in\CA_\cusp(\GL_{a_i})$ for $i=1,2,\cdots,r$, contains a cuspidal member.
An interesting problem is to figure out the explicit formula of the bounds $N_{\ul{a},\ul{b}}^{(1)}$ and $N_{\ul{a},\ul{b}}^{(2)}$
as functions of $\ul{a}$ and $\ul{b}$.
Secondly, one may easily write down the corresponding analogues of Theorem \ref{ncmain2} for bounds $N_{\ul{a},\ul{b}}^{(1)}$ and $N_{\ul{a},\ul{b}}^{(2)}$, we omit them here.
Finally, we give examples to indicate that $N_{\ul{a}} > N_{\ul{a},\ul{b}}^{(1)} > N_{\ul{a},\ul{b}}^{(2)}$.

Consider $\psi=(\tau_1, 1) \boxplus (\tau_2, 8)$, where $\tau_1\in\CA_\cusp(\GL_5)$ of orthogonal type, and $\tau_2\in\CA_\cusp(\GL_2)$
of symplectic type. By Remark \ref{rmkbv},
$$\eta(\ul{p}_{\psi})=(([1^58^2]^t)^-)_{\Sp}=[72^61]_{\Sp}=[62^7].$$
In this case, one has that $N_{\ul{a}}=(5+2)^2-1=48$. On the other hand, one has that $N_{\ul{a},\ul{b}}^{(1)}=24$ and $N_{\ul{a},\ul{b}}^{(2)}=16$.

In fact, $[4^42^4]$ is the only partition $\ul{p}$ that gives maximal $\lvert \ul{p} \rvert$, and satisfies the conditions:
$\ul{p} \leq \eta(\ul{p}_{\psi})$ under the lexicographical order of partitions, and $\ul{p}$ has the form
$[(2n_1)^{s_1}(2n_2)^{s_2} \cdots (2n_r)^{s_r}]$ with $2n_1 > 2n_2 > \cdots 2n_r$ and $s_i \leq 4$ for $1 \leq i \leq r$.
This shows that $N_{\ul{a},\ul{b}}^{(1)}=24$.

Also, $[4^22^4]$ is the only partition $\ul{p}$ that gives maximal $\lvert \ul{p} \rvert$, and satisfies the conditions:
$\ul{p} \leq \eta(\ul{p}_{\psi})$ under the dominance order of partitions, and $\ul{p}$ has the form
$[(2n_1)^{s_1}(2n_2)^{s_2} \cdots (2n_r)^{s_r}]$ with $2n_1 > 2n_2 > \cdots 2n_r$ and $s_i \leq 4$ for $1 \leq i \leq r$. This shows that $N_{\ul{a},\ul{b}}^{(2)}=16$.

Note that the bound $N_{\ul{a},\ul{b}}^{(1)}$ uses Theorem \ref{ub}, while the bound $N_{\ul{a},\ul{b}}^{(2)}$ needs the assumption that
Part (1) of Conjecture \ref{J14} holds.

\subsection{Examples}
We give examples of Arthur parameters $\psi$ such that $\wt{\Pi}_{\psi}(\Sp_{2n}) \cap \CA_2(\Sp_{2n})$ contains no cuspidal members.

{\bf Example 1:}\
Let $\tau\in\CA_\cusp(\GL_{2l})$ be such that $L(s, \tau, \wedge^2)$ has a pole at $s=1$. Consider the Arthur parameter $\psi=(\tau, 2m) \boxplus (1_{\GL_1(\BA)}, 1)$.
In this case, we have that $\ul{a}=(2l,1)$ and $\ul{b}=(2m,1)$. Since $a_1+a_2=2l+1$ is odd, we have that
$$
N_{\ul{a}}=(a_1+a_2)^2-1=(2l+1)^2-1.
$$
If $m > l+1$, then we have
$$
4ml=a_1b_1+a_2b_2-1=2l(2m)+1-1 > (2l+1)^2-1=N_{\ul{a}},
$$
and hence, by Theorem \ref{ncmain1} or Theorem \ref{ncmain2},
$\wt{\Pi}_{\psi}(\Sp_{4ml}) \cap \CA_2(\Sp_{4ml})$ contains no cuspidal members.

But, if in addition, $L(\frac{1}{2}, \tau) \neq 0$, we can construct a residual representation in $\wt{\Pi}_{\psi}(\Sp_{4ml}) \cap \CA_2(\Sp_{4ml})$ as follows.
Let $P_{2ml} = M_{2ml}N_{2ml}$ be the parabolic subgroup of $\Sp_{4ml}$ with Levi subgroup $M_{2ml} \cong \GL_{2l}^{\times m}$.
For any
$$\phi \in A(N_{2ml}(\BA)M_{2ml}(F) \bs \Sp_{4ml}(\BA))_{\Delta(\tau,m)},$$
following \cite{L76} and \cite{MW95}, a residual Eisenstein series can be defined by
$$
E(\phi,s)(g)=\sum_{\gamma\in P_{2ml}(F)\bks \Sp_{4ml}(F)}\lambda_s \phi(\gamma g).
$$
It converges absolutely for real part of $s$ large and has meromorphic continuation to the whole complex plane $\BC$. Since $L(\frac{1}{2}, \tau) \neq 0$, by \cite{JLZ13}, this Eisenstein series
has a simple pole at $\frac{m}{2}$, which is the right-most one.
Denote the representation generated by these residues at $s=\frac{m}{2}$
by $\CE_{\Delta(\tau, m)}$, which is square-integrable.
By \cite[Section 6.2]{JLZ13}, $\CE_{\Delta(\tau,m)}$ has the global Arthur parameter
$\psi=(\tau, 2m) \boxplus (1_{\GL_1(\BA)}, 1)$, and hence belongs to $\wt{\Pi}_{\psi}(\Sp_{4ml}) \cap \CA_2(\Sp_{4ml})$.

{\bf Example 2:}\
Consider a family of Arthur parameters of symplectic groups of the form
$$
\psi=(1_{\GL_1(\BA)}, b_1) \boxplus (\tau, b_2),
$$
where $b_1\geq 1$ is odd, $\tau\in\CA_\cusp(\GL_2)$ is of symplectic type and $b_2\geq 1$ is even.
By definition, $\ul{p}_{\psi} = [b_1b_2^2]$, and
$$\eta(\ul{p}_{\psi})=((\ul{p}_{\psi}^{-})_{\Sp})^t = ((\ul{p}_{\psi}^t)^{-})_{\Sp}=(([1^{b_1}]+[2^{b_2}])^{-})_{\Sp}.$$
It is clear that the biggest part occurring in the partition $\eta(\ul{p}_{\psi})$ is at most $3$.
Note that $2n=a_1b_1+a_2b_2-1=b_1+2b_2-1$.

Assume that $\pi$ belongs to $\wt{\Pi}_{\psi}(\Sp_{2n})\cap\CA_\cusp(\Sp_{2n})$ with the above given global Arthur parameter $\psi$.
By Theorem \ref{ub}, for any $\ul{p} \in \frak{p}^m(\pi)$, its biggest part is smaller than or equal to 3.
On the other hand, the partition $\ul{p}_{\pi} \in \frak{p}^m(\pi)$ constructed in \cite{GRS03} is even.
Hence, $\ul{p}_{\pi} = \{[2^n]\}$. Since $F$ is totally imaginary, by Theorem \ref{ti}, we must have that $n \leq 4$.
Hence, one can see that $N_{\ul{a}}=N_{\ul{a},\ul{b}}^{(1)}=N_{\ul{a},\ul{b}}^{(2)}=8$, where $\ul{a}=\{1,2\}$, $\ul{b}=\{b_1,b_2\}$.
It follows from Theorems \ref{ncmain1}--\ref{ncmain4}
that $\wt{\Pi}_{\psi}(\Sp_{2n}) \cap \CA_2(\Sp_{2n})$ contains
no cuspidal members except probably the following cases (see Figure 1 below)
$$
(b_1,b_2)=(1,2), (1,4), (3,2), (5,2).
$$
In particular, the global Arthur packet $\wt{\Pi}_{\psi}(\Sp_{2n})$ contains
no cuspidal members if $n\geq 5$.

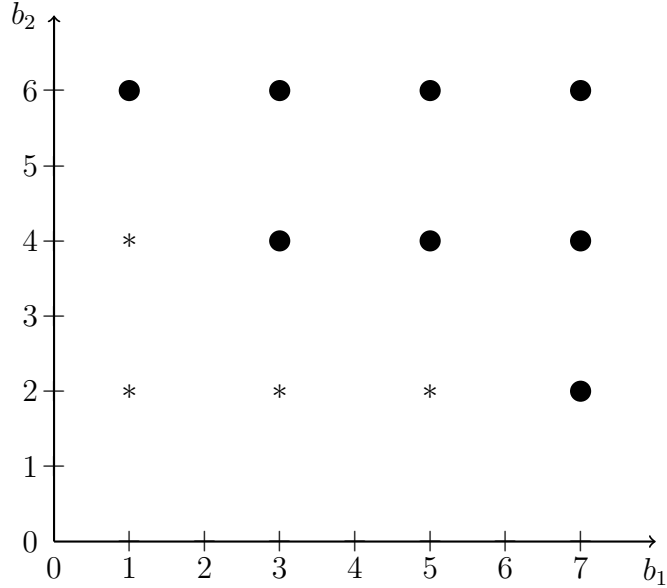
\begin{figure}[h]
\begin{tikzpicture}
   \tkzInit[xmax=7,ymax=6,xmin=0,ymin=0]
\tkzLabelX
\tkzLabelY
   \fill [black](7,2) circle(4pt);
   \fill [black](3,4) circle(4pt);
   \fill [black](5,4) circle(4pt);
   \fill [black](7,4) circle(4pt);
   \fill [black](1,6) circle(4pt);
   \fill [black](3,6) circle(4pt);
   \fill [black](5,6) circle(4pt);
   \fill [black](7,6) circle(4pt);
        \draw [thick,->] (0, 0) -- (8, 0);
        \draw [thick,->] (0, 0) -- (0, 7);
        \node at (8,-0.4) {$b_1$};
        \node at (-0.4, 7) {$b_2$};
        \node at (1,0) {$+$};
        \node at (3,0) {$+$};
        \node at (5,0) {$+$};
        \node at (2,0) {$+$};
        \node at (4,0) {$+$};
        \node at (6,0) {$+$};
        \node at (7,0) {$+$};
        \node at (0,2) {$+$};
        \node at (0,4) {$+$};
        \node at (0,6) {$+$};
        \node at (0,1) {$+$};
        \node at (0,3) {$+$};
        \node at (0,5) {$+$};
        \node at (1,2) {$*$};
        \node at (3,2) {$*$};
        \node at (5,2) {$*$};
        \node at (1,4) {$*$};
  \tkzText[above](4,-1){}
  \end{tikzpicture}
  \caption{Only *'s indicate that the Arthur packets possibly contain cuspidal members.}
\end{figure}

As we mentioned before that for generic global Arthur parameters $\phi\in\wt{\Phi}_2(\RG)$, one must have
$$
\wt{\Pi}_\phi(\RG)\cap\CA_2(\RG)\subset\CA_\cusp(\RG).
$$
In \cite{M08} and \cite{M11}, M{\oe}glin considers the problem on which non-generic global Arthur packets contains
non-cuspidal members, i.e. the square-integrable residual representations of $\RG(\BA)$. She gives
a conjecture on necessary and sufficient conditions for this problem and proves the conjecture when the square-integral representations with cohomology at infinity. Moreover, in \cite[Section 4.6]{M08}, M{\oe}glin predicts that her conjecture implies that for
a given global Arthur parameter $\psi=\boxplus_{i=1}^r (\tau_i, b_i)$ of a symplectic group $\Sp_{2n}$,
where $\tau_i\in\CA_\cusp(\GL_{a_i})$ is self-dual,
if there exist $1 \leq j_1 \leq r$ such that $b_{j_1} \geq a_{j_1}+a_{j_2}+b_{j_2}$, for any $1 \leq j_2 \neq j_1 \leq r$, then
$\wt{\Pi}_{\psi}(\Sp_{2n}) \cap \CA_2(\Sp_{2n})$ contains no cuspidal members.
Comparing to our discussions and examples above, one may easily find that
{\bf Example 1} gives examples that $\wt{\Pi}_{\psi}(\Sp_{2n}) \cap \CA_2(\Sp_{2n})$ contains no cuspidal members, which matches her prediction.
But, our {\bf Example 2} contains many more cases that $\wt{\Pi}_{\psi}(\Sp_{2n}) \cap \CA_2(\Sp_{2n})$ contains no cuspidal members, which can not be
determined by the condition suggested by M{\oe}glin.  We remark that {\bf Example 2} also includes cases that can not be decided
by the discussion in Section 3. One of such cases is that given by $(b_1,b_2)=(5,6)$.

\section{On generalized Ramanujan problem}

The generalized Ramanujan problem as proposed by P. Sarnak in \cite[Section 2]{Sar05} is to understand the behavior of the local components of
irreducible cuspidal automorphic representations of $\RG(\BA)$ for general reductive algebraic group $\RG$ defined over a number field $F$.
The generalized Ramanujan conjecture asserts that all local components of irreducible generic cuspidal representations are tempered.
When the group $\RG$ is not a general linear group, an irreducible cuspidal automorphic representation $\pi$ of $\RG(\BA)$ may have non-tempered
local components. Examples are those cuspidal members in a global Arthur packet with a non-generic global Arthur parameters. Hence it
is important also from this prospective to determine which non-generic global Arthur packets have no cuspidal members.

More precisely, the endoscopic classification of Arthur provides certain bounds for the exponents of the unramified local components of
the irreducible automorphic representations occurring in the discrete spectrum. It is clear that if one is able to determine
which non-generic global packets have no cuspidal members, the bounds of the exponents of the unramified local components of the cuspidal
spectrum would be much improved, which definitely helps us to the understanding of the generalized Ramanujan problem.

In this section, we take a preliminary step to understand the bounds of exponents of the unramified local components of the cuspidal
spectrum of $\Sp_{2n}$ based on the results obtained in Section 4.

For $\pi\in\CA_\cusp(\Sp_{2n})$ and $\theta \in \BR_{\geq 0}$, we say that $\pi$ satisfies $R(\theta)$ if each of its unramified components $\pi_v$ is the unique unramified component of the induced representation
$$\Ind_{B(F_v)}^{\Sp_{2n}(F_v)} \chi_1 \lvert \cdot \rvert^{\alpha_1} \otimes \chi_2 \lvert \cdot \rvert^{\alpha_2}
\otimes \cdots \otimes \chi_n \lvert \cdot \rvert^{\alpha_n},$$
where $B$ is the standard Borel subgroup of $\Sp_{2n}$, with the property that
for $1 \leq i \leq n$, $\chi_i$ are unitary unramified characters of $F_v^*$, such that $0 \leq \alpha_i \leq \theta$.

By the discussion in Remark \ref{rmkKR}, if there is a simple global Arthur parameter $(\chi, b)$ occurring as a formal summand
in the global Arthur parameter $\psi$ of $\pi$, one must have that $b \leq n+1$ if $n$ is even, and that $b \leq n$ if $n$ is odd,
where $\chi$ is a quadratic automorphic character of $\GL_1(\BA)$. In order to figure out an upper bound $\theta$ for
every $\pi\in\CA_\cusp(\Sp_{2n})$ to satisfy $R(\theta)$, one only needs to consider simple global Arthur parameters $(\tau,b)$ that
may occur in the global Arthur parameter $\psi$ of $\pi$, where $\tau\in\CA_\cusp(\GL_2)$ being self-dual.

First, assume that $n$ is even. Consider a global Arthur parameter of $\Sp_{2n}(\BA)$,
$\psi=(1_{\GL_1(\BA)},1) \boxplus (\tau, n)$, with $\tau\in\CA_\cusp(\GL_2)$ of symplectic type.
By using the bound of Kim-Sarnak (\cite{KS03}) and Blomer-Brumley (\cite{BB11})
towards the Ramanujan conjecture for $\GL_2$, which is $R(\frac{7}{64})$,
one may easily figure out that any $\pi\in\CA_\cusp(\Sp_{2n})\cap\wt{\Pi}_\psi(\Sp_{2n})$ satisfies $R(\frac{7}{64}+\frac{n-1}{2})$.
By the result of Kudla and Rallis (\cite{KR94}), for any $\pi\in\CA_\cusp(\Sp_{2n})\cap\wt{\Pi}_\psi(\Sp_{2n})$
(with $n$ even), if a simple global Arthur parameter $(\chi,b)$ occurs in the global Arthur parameter $\psi$ of $\pi$,
one must have that $b$ is at most $n+1$, and hence satisfies $R(\frac{n}{2})$.
Note that $\frac{7}{64}+\frac{n-1}{2} < \frac{n}{2}$. It follows that $\frac{n}{2}$ is a possible upper bound for all
$\pi\in\CA_\cusp(\Sp_{2n})$.
On the other hand, Piatetski-Shapiro and Rallis (\cite{PSR88})
construct a cuspidal member $\pi\in\wt{\Pi}_\psi(\Sp_{2n})$ (with $n$ even) that has the simple global Arthur parameter $(\chi, n+1)$ occurring in
the $\psi$. Therefore, we obtain that $\frac{n}{2}$ is the sharp upper bound for all $\pi\in\CA_\cusp(\Sp_{2n})$ when $n$ is even.
We state the conclusion of the above discussion as
\begin{prop}\label{ubeven}
Let $F$ be a number field. When $n$ is an even integer, all $\pi\in\CA_\cusp(\Sp_{2n})$ satisfy $R(\frac{n}{2})$, and the bound $\frac{n}{2}$ is
achieved by the $\pi\in\CA_\cusp(\Sp_{2n})$ constructed by
Piatetski-Shapiro and Rallis in \cite{PSR88}.
\end{prop}

Next, assume that $n$ is odd. Consider a global Arthur parameter of $\Sp_{2n}(\BA)$,
$\psi=(\omega_{\tau},1) \boxplus (\tau, n)$, with $\tau\in\CA_\cusp(\GL_2)$ of orthogonal type and
$\omega_{\tau}$ the central character of $\tau$. By the same reason, one has that all $\pi\in\CA_\cusp(\Sp_{2n})\cap\wt{\Pi}_\psi(\Sp_{2n})$
satisfy $R(\frac{7}{64}+\frac{n-1}{2})$.
Again by \cite{KR94}, for any $\pi\in\CA_\cusp(\Sp_{2n})\cap\wt{\Pi}_\psi(\Sp_{2n})$
(with $n$ odd), if a simple global Arthur parameter $(\chi,b)$ occurs in the global Arthur parameter $\psi$ of $\pi$,
one must have that $b$ is at most $n$, and hence satisfies $R(\frac{n-1}{2})$.
Because $\frac{n-1}{2} < \frac{7}{64}+\frac{n-1}{2}$, we obtain that $\frac{7}{64}+\frac{n-1}{2}$ is a possible upper bound for any
$\pi\in\CA_\cusp(\Sp_{2n})$.

However, by Theorem \ref{ncmain1}, if we assume that $F$ is totally imaginary and $n \geq 5$, then for the Arthur parameters
$\psi=(\omega_{\tau},1) \boxplus (\tau, n)$ given above, there does not exist any cuspidal member in
$\wt{\Pi}_\psi(\Sp_{2n}) \cap \CA_2(\Sp_{2n})$. Hence, we obtain the following conclusion.

\begin{prop}\label{ubodd}
Assume that $F$ is totally imaginary and $n \geq 5$ is odd. Any $\pi\in\CA_\cusp(\Sp_{2n})$ satisfies $R(\frac{n-1}{2})$.
\end{prop}

We may expect that a simple global Arthur parameter $(\tau,n-1)$ with $n$ odd and $\tau\in\CA_\cusp(\GL_2)$ of symplectic type could
have cuspidal members in the global Arthur packet $\wt{\Pi}_\psi(\Sp_{2n})$, although we do not know how to construct them for the moment.
However, in the case the bound is $\frac{7}{64}+\frac{n-2}{2}$, which is less than $\frac{n-1}{2}$.
Also, for $\tau\in\CA_\cusp(\GL_a)$ (self-dual) with $a\geq 3$, the simple global Arthur parameters of type $(\tau, b)$ produce naturally
a bound better than what obtained above, and hence are omitted for further consideration.

It is a very interesting problem to determine the sharp upper bound $\theta$ for the cuspidal spectrum of $\Sp_{2n}(\BA)$ when $n$ is odd. This would involve a generalization or extension of the constructions by Piatetski-Shapiro and Rallis (\cite{PSR88}) and by Ikeda (\cite{Ik01} and \cite{Ik}).  We will get back to this issue in our future work.

\section{Small Cuspidal Automorphic Representations}\label{SCAR}

In this section, we discuss some criteria on the {\sl smallness} of cuspidal automorphic representations of $\Sp_{2n}(\BA)$ and
gives examples of small cuspidal automorphic representations, in addition to the examples constructed by Ikeda in \cite{Ik01}.
From now on, we assume that $F$ is a number field.

\subsection{Characterization of small cuspidal representations}
The characterization of small cuspidal automorphic representations will be given in terms of a vanishing condition on Fourier coefficients
related to the automorphic descent method (\cite{GRS11}), and also in terms of the notion of hyper-cuspidality in the sense of
Piatetski-Shapiro (\cite{PS83}). Also, our discussions cover the case of symplectic group $\Sp_{2n}(\BA)$ and the case of
the metaplectic double cover $\wt{\Sp}_{2n}(\BA)$ of $\Sp_{2n}(\BA)$ together.

\begin{thm}\label{spclt}
Assume that $\pi$ is an irreducible cuspidal automorphic  representation of $\Sp_{2n}(\BA)$ or $\wt{\Sp}_{2n}(\BA)$. Then $\frak{p}^m(\pi) = \{[2^n]\}$ if and only if $\pi$ has no nonzero Fourier coefficients attached to the partition $[41^{2n-4}]$.
\end{thm}

\begin{proof}
First, assume that $\frak{p}^m(\pi) = \{[2^n]\}$. In this case, the even partition $\ul{p}_{\pi}$ constructed in \cite{GRS03} is exactly $[2^n]$. Note that $\ul{p}_{\pi}$ has the property that ``maximal at every stage" (see the proof of \cite[Theorem 2.7]{GRS03} or \cite[Remark 5.1]{JL15}), which implies directly that $\pi$ has no nonzero Fourier coefficients attached to the partition $[41^{2n-4}]$.

Next, assume that $\pi$ has no nonzero Fourier coefficients attached to the partition $[41^{2n-4}]$.
By Lemma \ref{spvandescent} below, $\pi$ has no nonzero Fourier coefficients attached to the partition
$[(2k)1^{2n-2k}]$, for any $2 \leq k \leq n$. Assume that $\ul{p}=[p_1p_2\cdots p_s] \in \frak{p}^m(\pi)$,
with $p_1 \geq p _2 \geq \cdots \geq p_s$.
If $p_1$ is odd, then one must have that $p_1 \geq 3$. By \cite[Lemma 3.3]{JL15},
$\pi$ has a nonzero Fourier coefficient attached to the partition $[(p_1)^21^{2n-2p_1}]$.
Then \cite[Lemma 2.4]{GRS03} shows that $\pi$ must have a nonzero Fourier coefficient attached to the partition
$[(2r)1^{2n-2r}]$ for some $2r > 2p_1 \geq 6$, which contradicts the assumption of the theorem.
Now, if $p_1$ is even, then
by \cite[Lemma 2.6]{GRS03} or \cite[Lemma 3.1]{JL15}, $\pi$ has a nonzero Fourier coefficient attached to the partition $[(p_1)1^{2n-p_1}]$. By assumption of the theorem, we must have that $p_1=2$. Hence we obtain that $2 = p_1 \geq p_2 \geq \cdots \geq p_s$, which implies that
$\ul{p} \leq [2^n]$.
On the other hand, by Theorem \ref{li}, the cuspidal $\pi$ must have a nonzero Fourier coefficient attached to the partition $[2^n]$.
It follows that for any $\ul{p}\in\frak{p}^m(\pi)$, the case that $\ul{p} < [2^n]$ can not happen.
Therefore, we conclude that $\ul{p} = [2^n]$, and hence $\frak{p}^m(\pi) = \{[2^n]\}$.
This completes the proof of the theorem.
\end{proof}

Let $\alpha = e_1 - e_{2n}$ be the longest positive root of $\Sp_{2n}$ and let $X_{\alpha}$ be the corresponding one-dimensional root subgroup.
Recall from \cite[Section 6]{PS83} that an automorphic function $\varphi$ is called {\it hypercuspidal} if
$$\int_{X_{\alpha}(F) \bs X_{\alpha}(\BA)} \varphi(xg) dx \equiv 0.$$ It is clear that any hypercupsidal function is automatically cuspidal. An automorpic representation $\pi$ of $\Sp_{2n}(\BA)$ or $\wt{\Sp}_{2n}(\BA)$ is called {\it hypercuspidal} if every $\varphi \in \pi$ is hypercuspidal.

For $0 \leq i \leq n-1$, let $P_i=M_iN_i$ be the parabolic subgroup
of $\Sp_{2n}$ with Levi subgroup $M \cong \GL_1^i \times \Sp_{2n-2i}$. Define a character of $N_i$ by $\psi_i(n)=\psi(\sum_{j=1}^{i} n_{j,j+1})$.
Let $\pi$ be an automorphic representation of $\Sp_{2n}(\BA)$ or $\wt{\Sp}_{2n}(\BA)$. For any $\varphi \in \pi$, let
$$\CF_i(\varphi)(g)=\int_{N_i(F) \bs N_i(\BA)} \varphi(ng) \psi_i^{-1}(n) dn.$$

\begin{lem}\label{spfe}
Let $\pi$ be a cuspidal automorphic representation of $\Sp_{2n}(\BA)$ or $\wt{\Sp}_{2n}(\BA)$. For any $\varphi \in \pi$, $\CF_i(\varphi)$ is a linear combination of $\CF_{i+1}(\varphi)$ and
Fourier coefficients attached to the partition $[(2i+2)1^{2n-2i-2}]$.
\end{lem}

\begin{proof}
Let $\alpha$ be the root $e_{i+1}-e_{2n-i}$ and let $X_{\alpha}$ be the corresponding one-dimensional root subgroup. Since $X_{\alpha}$ normalizes $N_i$ and preserves the character $\psi_i$, one can take the Fourier expansion of $\CF_i(\varphi)$ along $X_{\alpha}(F) \bs X_{\alpha}(\BA)$. The non-constant terms give us exactly Fourier coefficients attached to the partition $[(2i+2)1^{2n-2i-2}]$. Now consider the constant term, that is $\int_{X_{\alpha}(F) \bs X_{\alpha}(\BA)} \CF_i(\varphi)(xg)dx$.

For $i+2 \leq j \leq 2n-i-1$, let $\alpha_j$ be the root $e_{i+1}-e_j$, and let $X_{\alpha_j}$ be the corresponding one-dimensional root subgroup. Let $X=\prod_{j=i+2}^{2n-i-1} X_{\alpha_j}$. Then, one can see that $X$ normalizes $N_iX_{\alpha}$ and preserves the character $\psi_i$. Here $\psi_i$ is extended trivially to $N_iX_{\alpha}$. Hence, one can take the Fourier expansion of $\int_{X_{\alpha}(F) \bs X_{\alpha}(\BA)} \CF_i(\varphi)(xg)dx$ along $X(F) \bs X(\BA)$, and obtain that
\begin{align*}
\ & \int_{X_{\alpha}(F) \bs X_{\alpha}(\BA)} \CF_i(\varphi)(xg)dx\\
= \ &  \sum_{\xi \in X(F)} \int_{X(F) \bs X(\BA)} \int_{X_{\alpha}(F) \bs X_{\alpha}(\BA)} \CF_i(\varphi)(xx'g)\psi_{\xi}^{-1}(x')dxdx'.
\end{align*}

Note that the constant term corresponding to $\xi=0$ is identically zero, since $\varphi \in \pi$ is cuspidal. Also note that $\Sp_{2n-2i-2}(F)$ acts on $X(F) \bs \{0\}$ transitively, and one can take a representative $\xi_0 =(1,0,\ldots,0)$. Denote the stabilizer of $\xi_0$ in $\Sp_{2n-2i-2}(F)$ by $H(F)$, which is a Jacobi group $\CH_{2n-2i-4}(F) \rtimes \Sp_{2n-2i-4}(F)$. Embed $\Sp_{2n-2i-2}$ into $\Sp_{2n}$ via $g \rightarrow \begin{pmatrix}
I_{i+1} & 0 & 0\\
0& g & 0\\
0 & 0 & I_{i+1}
\end{pmatrix}$, and identify it with its image under this embedding. Then the above Fourier expansion can be rewritten as
\begin{align*}
\ & \int_{X_{\alpha}(F) \bs X_{\alpha}(\BA)} \CF_i(\varphi)(xg)dx\\
= \ &  \sum_{\gamma \in H(F) \bs \Sp_{2n-2i-2}(F)} \int_{X(F) \bs X(\BA)} \int_{X_{\alpha}(F) \bs X_{\alpha}(\BA)} \CF_i(\varphi)(xx'\gamma g)\psi_{\xi_0}^{-1}(x')dxdx',
\end{align*}
which is exactly
$$\sum_{\gamma \in H(F) \bs \Sp_{2n-2i-2}(F)} \CF_{i+1}(\varphi)(\gamma g).$$
Therefore, $\CF_i(\varphi)$ is a linear combination of $\CF_{i+1}(\varphi)$ and Fourier coefficients attached to the partition $[(2i+2)1^{2n-2i-2}]$.
This completes the proof of the lemma.
\end{proof}

The following lemma is \cite[Key Lemma 3.3]{GRS05}, which has been used in the proof of Theorem \ref{spclt}. We state it here and give
a shorter proof, using Theorem \ref{ggsglobal}.

\begin{lem}[Key Lemma 3.3, \cite{GRS05}]\label{spvandescent}
Let $\pi$ be any automorphic representation of $\RG(\BA)=\Sp_{2n}(\BA)$ or $\wt{\Sp}_{2n}(\BA)$. If $\pi$ has no nonzero Fourier coefficients attached to the partition $[(2k)1^{2n-2k}]$, then $\pi$ has no nonzero Fourier coefficients attached to the partition $[(2k+2)1^{2n-2k-2}]$.
\end{lem}

\begin{proof}
We recall that the $F$-rational nilpotent orbits $\CO$ of $\frak{sp}_{2n}(F)$ corresponding to the partition $[(2k+2)1^{2n-2k-2}]$
are parameterized by square classes $\beta \in F^* / (F^*)^2$.
For $1 \leq j \leq k$, define the root $\alpha_j=e_{j+1}-e_j$.
Let $\alpha_{k+1}=e_{2n+1-k-1}-e_{k+1}$.
For $1 \leq j \leq k+1$, let $x_{\alpha_j}$ be the corresponding root subspace in the Lie algebra.
A representative of $\CO$ can be chosen to be $f=\sum_{j=1}^{k} x_{\alpha_j}(\frac{1}{2})+x_{\alpha_{k+1}}(\beta)$. It is clear that $f$ can be decomposed as $f_1 +f_2$, where $f_1 =
x_{\alpha_1}(\frac{1}{2})$, and $f_2=\sum_{j=2}^{k} x_{\alpha_j}(\frac{1}{2})+x_{\alpha_{k+1}}(\beta)$. Note that the nilpotent orbit containing $f_2$ corresponds to the partition $[(2k)1^{2n-2k}]$.

For $f$ and $f_i$, $i=1,2$, By Jacobson-Morozov Theorem, there exist $\frak{sl}_2$-triples $(e,h,f)$, $(e_i,h_i,f_i)$, such that $[h,u]=-2u$, $[h_i, u_i]=-2u_i$. Let $\RG(\BA)_h$ be the centralizer of $h$ in $\RG(\BA)$, which contains the maximal split torus $\RT(\BA)$ of $\RG(\BA)$. It is clear that $f_2 \in \ol{\RG(\BA)_h f}$. Indeed, take $t = \diag(t_1, 1, \ldots, 1, t_1^{-1})$ with $t_1^{-1} \rightarrow 0$, then $t \cdot f \rightarrow f_2$.
By Theorem \ref{ggsglobal}, if $\pi$ has a nonzero Fourier coefficient attached to $f$, then it has a nonzero Fourier coefficient attached to $f_2$. Since by assumption, $\pi$ has no nonzero Fourier coefficients attached to $[(2k)1^{2n-2k}]$, we can conclude that $\pi$ also has no nonzero Fourier coefficients attached to the partition $[(2k+2)1^{2n-2k-2}]$.
\end{proof}

\begin{thm}\label{spclt2}
For an irreducible cuspidal automorphic  representation $\pi$ of $\Sp_{2n}(\BA)$ or $\wt{\Sp}_{2n}(\BA)$,
$\frak{p}^m(\pi) = \{[2^n]\}$ if and only if $\pi$ is hypercuspidal.
\end{thm}

\begin{proof}
By Theorem \ref{spclt}, we just need to show that $\pi$ is hypercuspidal if and only if $\pi$ has no nonzero Fourier coefficients
attached to partition $[41^{2n-4}]$. First, it is clear that if $\pi$ is hypercuspidal, then $\pi$ has no nonzero Fourier coefficients
attached to partition $[41^{2n-4}]$, since $X_{\alpha}$， for the longest root $\alpha$, is the center of the standard maximal
unipotent subgroup of $\Sp_{2n}$. Now assume that $\pi$ has no nonzero Fourier coefficients attached to partition $[41^{2n-4}]$.
By Lemma \ref{spvandescent}, $\pi$ has no nonzero Fourier coefficients attached to partition $[(2k)1^{2n-2k}]$, for any $2 \leq k \leq n$.

Let $Y$ be the unipotent subgroup of $\Sp_{2n}$ consisting of elements $y=\begin{pmatrix}
1 & x & *\\
0 & I_{2n-2} & x^*\\
0 & 0 & 1
\end{pmatrix}$, where $x \in \Mat_{1 \times (2n-2)}$. It is clear that $Y$ normalizes $X_{\alpha}$. Hence, $f(g):=\int_{X_{\alpha}(F) \bs X_{\alpha}(\BA)} \phi(xg) dx$ can be viewed as an automorphic function over $Y(F) \bs Y(\BA)$. After taking Fourier expansion along $Y(F) \bs Y(\BA)$,
\begin{equation}\label{spclt2equ1}
f(g)=\sum_{\xi \in F^{2n-2}\bs \{0\}}
\int_{Y(F) \bs Y(\BA)} f(yg)\psi_{\xi}^{-1}(y) dy,
\end{equation}
since $\pi$ is a cuspidal.

Note that the action of $\Sp_{2n-2}(F)$ on $F^{2n-2}\bs \{0\}$ via conjugation is transitive. Take a representative $\xi_0=(1,0,\ldots,0)$. Then its stabilizer in $\Sp_{2n-2}(F)$ is a subgroup (denoted by $H$) consisting of elements $\begin{pmatrix}
1 & x & y \\
0 & g' & x^*\\
0 & 0 & 1
\end{pmatrix}$, where $x \in \Mat_{1 \times {2n-4}}$, $y \in F$, $g' \in \Sp_{2n-4}$. Embed $\Sp_{2n-2}$ into $\Sp_{2n}$ via the map $g \rightarrow \begin{pmatrix}
1 & 0 & 0 \\
0 & g & 0\\
0 & 0 & 1
\end{pmatrix}$, and identify $\Sp_{2n-2}$ with its image under this embedding.
Then, after changing of variables, the Fourier expansion in \eqref{spclt2equ1} can be rewritten as
\begin{equation}\label{spclt2equ2}
f(g)=\sum_{\gamma \in H \bs \Sp_{2n-2}(F)}
\int_{Y(F) \bs Y(\BA)} f(y\gamma g)\psi_{\xi_0}^{-1}(y) dy,
\end{equation}
which is exactly $\sum_{\gamma \in H \bs \Sp_{2n-2}(F)} \CF_1(f)(\gamma g)$.
Hence, to show that $f$ is identically zero, it is enough to show that $\CF_1(f)$ is identically zero.

Applying Lemma \ref{spfe} repeatedly, $\CF_1(f)$ is a linear combination of Fourier coefficients
attached to the partitions $[(2k)1^{2n-2k}]$,
$2 \leq k \leq n$, which are all identically zero, by the above discussion. Therefore, $f$ is identically zero, i.e., $\pi$ is hypercuspidal.

This completes the proof of the theorem.
\end{proof}

Combining Theorems \ref{spclt}, \ref{spclt2} with Theorem \ref{ti}, we have the following corollary.

\begin{thm}\label{nohc}
Assume that $F$ is a totally imaginary number field and $n \geq 5$.
Then $\Sp_{2n}(\BA)$ and $\wt{\Sp}_{2n}(\BA)$ have no cuspidal representations having nonzero Fourier coefficients attached to the partitions $[41^{4n-4}]$, and equivalently, have no nonzero hypercuspidal representations.
\end{thm}

\begin{proof}
Assume that $\Sp_{2n}(\BA)$ and $\wt{\Sp}_{2n}(\BA)$ has a nonzero cuspidal representation $\pi$ having nonzero Fourier coefficients attached to the partitions $[41^{4n-4}]$, equivalently, $\pi$ is hypercuspidal. Then, by Theorems \ref{spclt}, \ref{spclt2}, $\frak{p}^m(\pi) = \{[2^n]\}$. In particular, the even partitions $\ul{p}_{\pi}$ constructed in \cite{GRS03} is exactly $[2^n]$. On the other hand, since $F$ is totally imaginary, by Theorem \ref{ti}, $\ul{p}_{\pi}$ can not be $[2^n]$ because of $n \geq 5$.
Contradiction.
\end{proof}

\subsection{Examples of small cuspidal representations}

In this section, we assume that $F$ is not totally imaginary number field if $n \geq 5$. In order to provide examples of global Arthur packets of $\Sp_{2n}$ whose cuspidal automorphic members $\pi$ have the property that
$\Fp^m(\pi) = \{[2^n]\}$, we separate the discussion according the parity of the integer $n$.

\subsubsection{{\bf Case of $n=2e$.}}

\begin{prop}\label{speven2}
Any $\pi\in\wt{\Pi}_\psi(\Sp_{4e}) \cap\CA_\cusp(\Sp_{4e})$ with
$$
\psi=(\tau, 2i) \boxplus (1_{\GL_1(\BA)}, 4e-4i+1), e \leq 2i \leq 2e,
$$
and $\tau \in \CA_{\cusp}(\GL_2)$ of symplectic type, has the property that $\frak{p}^m(\pi) = \{[2^{2e}]\}$,
 and hence is small.
\end{prop}

\begin{proof}
For $\psi=(\tau, 2i) \boxplus (1_{\GL_1(\BA)}, 4e-4i+1)$, with $e \leq 2i \leq 2e$, we must have that
$\ul{p}_{\psi}=[(2i)^2 (4e-4i+1)]$ and $\eta(\ul{p}_{\psi})$ has largest part at most $3$.
Any $\pi\in\CA_\cusp(\Sp_{4e})\cap\wt{\Pi}_\psi(\Sp_{4e})$,
by Theorem \ref{ub}, any partition $\ul{p} \in \frak{p}^m(\pi)$ satisfies the property that
$\ul{p} \leq \ul{p}_{\psi}$ under the lexicographical order of partitions. Hence, any partition $\ul{p}=[p_1p_2 \cdots p_r] \in \frak{p}^m(\pi)$ has largest part $p_1 \leq 3$.

If $p_1=3$, then by \cite[Lemma 3.3]{JL15}, $\pi$ has a nonzero Fourier coefficient attached to the partition $[(p_1)^2 1^{4e-2p_1}]$. Furthermore,
by \cite[Lemma 2.4]{GRS03}, $\pi$ has a nonzero Fourier coefficient attached to the partition $[(2r)1^{4e-2r}]$ for some $2r > p_1 = 3$, which contradicts Theorem \ref{ub}. Hence we may have to take that $p_1=2$ and $\ul{p} \leq [2^{2e}]$ under the dominance order of partitions.
In this case, by Theorem \ref{li}, $\pi$ is non-singular. It follows again that any $\ul{p} \in \frak{p}^m(\pi)$ satisfies
the property that $\ul{p} \geq [2^{2e}]$ under the dominance order of partitions.
Therefore, we must have that $\frak{p}^m(\pi) = \{[2^{2e}]\}$.
\end{proof}

Note that if $2i < e$, then $4e-4i+1 > 2e+1$. By Remark \ref{rmkKR}, the global Arthur packet $\wt{\Pi}_\psi(\Sp_{4e})$ corresponding to
the global Arthur parameter
$$
\psi=(\tau, 2i) \boxplus (1_{\GL_1(\BA)}, 4e-4i+1)
$$
contains no cuspidal automorphic representations.

In the case of $2i=2e$, $\psi=(\tau, 2e) \boxplus (1_{\GL_1(\BA)}, 1)$, where $\tau \in \CA_{\cusp}(\GL_2)$ is of symplectic type.
If in addition $L(\frac{1}{2}, \tau) \neq 0$, then we can construct a residual representation in $\wt{\Pi}_{\psi}(\Sp_{4e})\cap\CA_2(\Sp_{4e})$ as follows.

Let $\Delta(\tau, e)$ be a Speh residual representation in the discrete spectrum of
$\GL_{2e}(\BA)$. For more information about the Speh residual representations, we refer to \cite{MW89}, or
\cite[Section 1.2]{JLZ13}. Let $P_r=M_rN_r$ be the maximal parabolic subgroup of
$\Sp_{2l}$ with Levi subgroup $M_r$ isomorphic to
$\GL_r \times \Sp_{2l-2r}$. Using
the normalization in \cite{Sh10},
the group $X_{M_{r}}^{\Sp_{2l}}$ of all continues homomorphisms from
$M_{r}(\BA)$ to $\BC^{\times}$, which is trivial
on $M_{r}(\BA)^1$ (see \cite{MW95}),
will be identified with $\BC$ by $s \rightarrow \lambda_s$.

For any $\phi \in A(N_{2e}(\BA)M_{2e}(F) \bs \Sp_{4e}(\BA))_{\Delta(\tau,e)}$, following \cite{L76} and \cite{MW95}, a
residual Eisenstein series
can be defined by
$$
E(\phi,s)(g)=\sum_{\gamma\in P_{2e}(F)\bks \Sp_{4e}(F)}\lambda_s \phi(\gamma g).
$$
It converges absolutely for real part of $s$ large and has meromorphic continuation to the whole complex plane $\BC$. Since $L(\frac{1}{2}, \tau) \neq 0$, by \cite{JLZ13}, this Eisenstein series has a simple pole at $\frac{e}{2}$, which is the right-most one.
Denote by $\CE_{\Delta(\tau, e)}$ the representation generated by these residues at $s=\frac{e}{2}$.
This residual representation is square-integrable.
By \cite[Section 6.2]{JLZ13}, the global Arthur parameter of $\CE_{\Delta(\tau,e)}$ is
$\psi=(\tau, 2e) \boxplus (1_{\GL_1(\BA)}, 1)$.
Hence $\CE_{\Delta(\tau,e)} \in \wt{\Pi}_{\psi}(\Sp_{4e})\cap\CA_2(\Sp_{4e})$.

%

By \cite[Theorem 1.3]{L13}, $\frak{p}^m(\CE_{\Delta(\tau,e)}) = \{[2^{2e}]\}$. For $\psi$ above, $\ul{p}_{\psi}= [(2e)^21]$ and $\eta(\ul{p}_{\psi})=[2^{2e}]$.
Hence, as mentioned in \cite{L13}, combining with Theorem \ref{ub},  all parts of Conjecture \ref{J14} have been proved for the Arthur parameter $\psi=(\tau, 2e) \boxplus (1_{\GL_1(\BA)}, 1)$ above.

\subsubsection{{\bf Case of $n=2e+1$.}}

\begin{prop}\label{spodd2}
Any $\pi\in\wt{\Pi}_{\psi}(\Sp_{4e+2})\cap\CA_\cusp(\Sp_{4e+2})$ with $\psi=(\tau, 2i+1) \boxplus (\omega_{\tau}, 4e-4i+1)$, $e \leq 2i \leq 2e$, and $\tau \in \CA_{\cusp}(\GL_2)$ of orthogonal type, has the property that $\frak{p}^m(\pi) = \{[2^{2e+1}]\}$, and hence is small.
\end{prop}

The proof of this proposition is similar to that of Proposition \ref{speven2}, and is omitted here. Note that if $2i < e$, then $4e-4i+1 > 2e+1$.
By Remark \ref{rmkKR}, the global Arthur packet $\wt{\Pi}_\psi(\Sp_{4e+2})$ associated to the global Arthur parameter
$$
\psi=(\tau, 2i+1) \boxplus (1_{\GL_1(\BA)}, 4e-4i+1)
$$
contains no cuspidal automorphic representations.

In the case of $2i=2e$, we can also construct a residual representation in $\wt{\Pi}_{\psi}(\Sp_{4e+2})\cap\CA_2(\Sp_{4e+2})$ as follows.

Since $\tau \in \CA_{\cusp}(\GL_2)$ is of orthogonal type, by the theory of automorphic descent of Ginzburg, Rallis and Soudry, there is a cuspidal representation $\pi'$ of $\SO_2^{\alpha}(\BA)$ which is anisotropic, such that $\pi'$ lifts to $\tau$ by the automorphic induction. Assume that there is an irreducible generic cuspidal representation $\pi$ of $\Sp_2(\BA)$ corresponding to $\pi'$ under the theta correspondence. Then the global Langlands functorial transfer from
$\Sp_2$ to $\GL_3$ takes $\pi$ to $\tau\boxplus 1$.

For any $\phi \in A(N_{2e}(\BA)M_{2e}(F) \bs \Sp_{4e+2}(\BA))_{\Delta(\tau,e) \otimes \pi}$, a
residual Eisenstein series can be defined as before by
$$
E(\phi,s)(g)=\sum_{\gamma\in P_{2e}(F)\bks \Sp_{4e+2}(F)}\lambda_s \phi(\gamma g).
$$
It converges absolutely for real part of $s$ large and has meromorphic continuation to the whole complex plane $\BC$. By \cite{JLZ13}, this Eisenstein series
has a simple pole at $\frac{e+1}{2}$, which is the right-most one.
Denote by $\CE_{\Delta(\tau, e) \otimes \pi}$ the representation generated by these residues at $s=\frac{e+1}{2}$.
This residual representation is square-integrable.
By \cite[Section 6.2]{JLZ13}, the global Arthur parameter of $\CE_{\Delta(\tau,e) \otimes \pi}$ is
$\psi=(\tau, 2e+1) \boxplus (\omega_{\tau}, 1)$. Hence $\CE_{\Delta(\tau,e) \otimes \pi} \in \wt{\Pi}_{\psi}(\Sp_{4e+2})\cap\CA_2(\Sp_{4e+2})$.


By \cite[Theorem 2.1]{JL15c}, $\frak{p}^m(\CE_{\Delta(\tau,e) \otimes \pi}) = \{[2^{2e+1}]\}$.
For $\psi=(\tau, 2e+1) \boxplus (\omega_{\tau}, 1)$ above, $\ul{p}_{\psi}=[(2e+1)^21]$ and $\eta(\ul{p}_{\psi})=[2^{2e+1}]$.
Hence, combining with Theorem \ref{ub}, all parts of Conjecture \ref{J14} have been proved for the Arthur parameter $\psi=(\tau, 2e+1) \boxplus (\omega_{\tau}, 1)$ above.

\subsubsection{{\bf Case $n=3e+1$.}}\
\begin{prop}\label{spodd3}
Any $\pi\in\wt{\Pi}_{\psi}(\Sp_{6e+2})\cap\CA_\cusp(\Sp_{6e+2})$ with $\psi=(\tau, 2e+1)$, and $\tau \in \CA_{\cusp}(\GL_3)$ of orthogonal type and with trivial central character,  has the property that $\frak{p}^m(\pi) = \{[2^{3e+1}]\}$, and hence is small.
\end{prop}

\begin{proof}
For $\psi=(\tau, 2e+1)$, we must have that $\ul{p}_{\psi}=[(2e+1)^3]$ and $\eta(\ul{p}_{\psi})=[3^{2e}2]$.
Take any $\pi\in\wt{\Pi}_{\psi}(\Sp_{6e+2})\cap\CA_\cusp(\Sp_{6e+2})$.
By Theorem \ref{ub}, for any $\ul{p}=[p_1 p_2 \cdots p_r] \in \frak{p}^m(\pi)$, we have that $\ul{p} \leq [3^{2e}2]$ under the lexicographical order of partitions. It follows that
$$
3 \geq p_1 \geq \cdots \geq p_r.
$$
If $p_1=3$, then by \cite[Lemma 3.3]{JL15}, $\pi$ has a nonzero Fourier coefficient attached to the partition $[(p_1)^2 1^{6e+2-2p_1}]$. Then, by \cite[Lemma 2.4]{GRS03}, $\pi$ has a nonzero Fourier coefficient attached to the partition $[(2r)1^{6e+2-2r}]$ for some $2r > p_1 = 3$, which contradicts Theorem \ref{ub}. Hence $p_1=2$, and $\ul{p} \leq [2^{3e+1}]$ under the dominance order of partitions.
On the other hand, by Theorem \ref{li}, $\pi$ is non-singular. Hence, any $\ul{p} \in \frak{p}^m(\pi)$ also satisfies the property that $\ul{p} \geq [2^{3e+1}]$ under the dominance order of partitions.
Therefore, we have proved that $\frak{p}^m(\pi) = \{[2^{3e+1}]\}$.

This completes the proof of the proposition.
\end{proof}

We can also construct a residual representation in $\wt{\Pi}_{\psi}(\Sp_{6e+2})\cap\CA_2(\Sp_{6e+2})$ as follows.
Since $\tau \in \CA_{\cusp}(\GL_3)$ has trivial central character, and $L(s, \tau, \Sym^2)$ has a pole at $s=1$, by the theory of automorphic descent (\cite{GRS11}), there is an irreducible generic cuspidal automorphic representation $\pi$ of $\Sp_2(\BA)$ that
lifts to $\tau$.

For any $\phi \in A(N_{3e}(\BA)M_{3e}(F) \bs \Sp_{6e+2}(\BA))_{\Delta(\tau,e) \otimes \pi}$, a
residual Eisenstein series can also be defined by
$$
E(\phi,s)(g)=\sum_{\gamma\in P_{3e}(F)\bks \Sp_{6e+2}(F)}\lambda_s \phi(\gamma g).
$$
It converges absolutely for real part of $s$ large and has meromorphic continuation to the whole complex plane $\BC$. By \cite{JLZ13}, this Eisenstein series
has a simple pole at $\frac{e+1}{2}$, which is the right-most one.
Denote by $\CE_{\Delta(\tau, e) \otimes \pi}$ the representation generated by these residues at $s=\frac{e+1}{2}$.
This residual representation is square-integrable.
By \cite[Section 6.2]{JLZ13}, the global Arthur parameter of $\CE_{\Delta(\tau,e) \otimes \pi}$ is
$\psi=(\tau, 2e+1)$. Hence $\CE_{\Delta(\tau,e) \otimes \pi} \in \wt{\Pi}_{\psi}(\Sp_{6e+2})\cap\CA_2(\Sp_{6e+2})$.

For $\psi=(\tau, 2e+1)$ as above, $\ul{p}_{\psi}=[(2e+1)^3]$, and $\eta(\ul{p}_{\psi})=[3^{2e}2]$.
Hence, by Theorem \ref{ub}, for any $\pi\in\wt{\Pi}_{\psi}(\Sp_{6e+2})\cap\CA_\cusp(\Sp_{6e+2})$, we have that
for any $\ul{p}=[p_1 p_2 \cdots p_r] \in \frak{p}^m(\pi)$, $\ul{p} \leq [3^{2e}2]$ under the lexicographical order of partitions,
and hence, $\ul{p} \leq [3^{2e}2]$ under the dominance order of partitions also. By \cite[Theorem 2.1]{JL15c}, $\frak{p}^m(\CE_{\Delta(\tau,e) \otimes \pi}) = \{[3^{2e}2]\}$. Therefore, all parts of Conjecture \ref{J14}
have been proved for the global Arthur parameter $\psi=(\tau, 2e+1)$ as above.

\subsection{Small cuspidal representations over totally imaginary number fields}

In this section, let $F$ be a totally imaginary number field.
Assume that $\frak{p}^m(\pi)$ is a singleton. Then $\frak{p}^m(\pi)$ consists of exactly the partition $\ul{p}_{\pi}$ constructed in \cite{GRS03}.

Let $\ul{p}(\Sp_{2n}, F)$ be the smallest even partition of $2n$ of the form
$$
[(2n_1)^{s_1}(2n_2)^{s_2} \cdots (2n_r)^{s_r}],
$$
with $2n_1 > 2n_2 > \cdots > 2n_r$ and $s_i \leq 4$ for $1 \leq i \leq r$.

\begin{exmp}
$\ul{p}(\Sp_{8}, F) = [2^4]$, $\ul{p}(\Sp_{10}, F) = [42^3]$, $\ul{p}(\Sp_{12}, F) = [42^4]$, $\ul{p}(\Sp_{14}, F) = [4^22^3]$ and
$\ul{p}(\Sp_{26}, F) = [64^32^4]$.
\end{exmp}

\begin{thm}\label{lowerbound}
Let $F$ be a totally imaginary number field.
Assume that $\frak{p}^m(\pi)$ is a singleton for any $\pi\in\CA_\cusp(\Sp_{2n})$.
Then any partition $\ul{p} \in \frak{p}_{sm}^{\Sp_{2n}, F}$ is bigger than or equal to $\ul{p}(\Sp_{2n}, F)$.
\end{thm}


\begin{thebibliography}{abcdefg}
\bibitem[Ac03]{Ac03}
P. Achar,
{\it An order-reversing duality map for conjugacy classes in Lusztig's canonical quotient.}
Transform. Groups \textbf{8} (2003), no. 2, 107--145.

\bibitem[Ar13]{Ar13}
J. Arthur,
{\it The endoscopic classification of representations: Orthogonal and Symplectic groups}.
Colloquium Publication Vol. \textbf{61}, 2013, American Mathematical Society.

\bibitem[BV85]{BV85}
D. Barbasch and D. Vogan,
{\it Unipotent representations of complex semisimple groups.}
Ann. of Math. (2) \textbf{121} (1985), no. 1, 41--110.

\bibitem[BB11]{BB11}
V. Blomer and F. Brumley,
\textit{On the Ramanujan conjecture over number fields}. Ann. of Math. (2) \textbf{174} (2011), no. 1, 581--605.


\bibitem[CM93]{CM93}
D. Collingwood and W. McGovern,
{\it Nilpotent orbits in semisimple Lie algebras.}
Van Nostrand Reinhold Mathematics Series. Van Nostrand Reinhold Co., New York, 1993. xiv+186 pp.

\bibitem[DI96]{DI96}
W. Duke and \"O. Imamoglu,
{\it A converse theorem and the Saito-Kurokawa lift},
Internat. Math. Res. Notices \textbf{7} (1996), 347--355.


\bibitem[G06]{G06}
D. Ginzburg,
{\it Certain conjectures relating unipotent orbits to automorphic representations.}
Israel J. Math. {\bf 151} (2006), 323--355.


\bibitem[GRS03]{GRS03}
D. Ginzburg, S. Rallis and D. Soudry,
{\it On Fourier coefficients of automorphic forms of symplectic groups.}
Manuscripta Math. \textbf{111} (2003), no. 1, 1--16.

\bibitem[GRS05]{GRS05}
D. Ginzburg, S. Rallis and D. Soudry,
{\it Contruction of CAP representations for Symplectic groups using the descent method.}
Automorphic representations, L-functions and applications: progress and prospects, 193--224, Ohio State Univ. Math. Res. Inst. Publ., {\bf 11}, de Gruyter, Berlin, 2005.

\bibitem[GRS11]{GRS11}
D. Ginzburg, S. Rallis and D. Soudry,
{\it The descent map from automorphic representations of {${\rm GL}(n)$} to classical groups.} World Scientific, Singapore, 2011. v+339 pp.

\bibitem[GGS15]{GGS15}
R. Gomez, D. Gourevitch and S. Sahi,
{\it Generalized and degenerate Whittaker models.}
Preprint. 2015.


\bibitem[H81]{H81}
R. Howe,
{\it Automorphic forms of low rank}.
Noncommutative harmonic analysis and Lie groups, 211--248, Lecture Notes in Math., {\bf 880}, Springer, 1981.

\bibitem[HPS79]{HPS79}
R. Howe and I. Piatetski-Shapiro,
{\it A counterexample to the "generalized Ramanujan conjecture'' for (quasi-) split groups}.
Automorphic forms, representations and L-functions, Part 1, pp. 315--322, Proc. Sympos. Pure Math., {\bf 33}, Amer. Math. Soc., Providence, R.I., 1979.

\bibitem[Ik01]{Ik01}
T. Ikeda,
{\it On the lifting of elliptic cusp forms to Siegel cusp forms of degree $2n$}. Ann. of Math. (2) {\bf 154} (2001), no. 3, 641--681.

\bibitem[Ik]{Ik}
T. Ikeda,
{\it On the lifting of automorphic representations of $\PGL_2(\BA)$ to $\Sp_{2n}(\BA)$ or $\wt{\Sp}_{2n+1}(\BA)$ over a totally real field}. Preprint.

\bibitem[J14]{J14}
D. Jiang,
{\it Automorphic Integral transforms for classical groups I: endoscopy correspondences}.
Automorphic Forms: L-functions and related geometry: assessing the legacy of I.I. Piatetski-Shapiro.
Comtemp. Math. \textbf{614}, 2014, AMS.

\bibitem[JL13]{JL13}
D. Jiang and B. Liu,
{\it On Fourier coefficients of automorphic forms of ${\rm GL}(n)$}.
Int. Math. Res. Not. 2013 (17): 4029--4071.

\bibitem[JL15]{JL15}
D. Jiang and B. Liu,
{\it On special unipotent orbits and Fourier coefficients for automorphic forms on symplectic groups}.
J. Number Theory {\bf 146} (2015), 343--389.

\bibitem[JL15a]{JL15a}
Jiang, Dihua; Liu, Baiying
{\it Fourier coefficients for automorphic forms on quasisplit classical groups}.
To appear in Cogdell volume. 2015.

\bibitem[JL15b]{JL15b}
D. Jiang and B. Liu,
{\it Arthur parameters and Fourier coefficients for automorphic forms on symplectic groups.}
To appear in Annales de l'Institut Fourier. 2015.

\bibitem[JL15c]{JL15c}
D. Jiang and B. Liu,
{\it On Fourier coefficients of certain residual representations of symplectic groups.}
To appear in Pacific J. of Math. 2015.

\bibitem[JLS15]{JLS15}
D. Jiang, B. Liu, and G. Savin,
{\it Raising nilpotent orbits for classical groups.}
Submitted. 2015.

\bibitem[JLXZ]{JLXZ}
D. Jiang, B. Liu, B. Xu, and L. Zhang,
{\it The Jacquet-Langlands correspondence via twisted descent}. accepted by IMRN 2015.

\bibitem[JLZ13]{JLZ13}
D. Jiang, B. Liu and L. Zhang,
{\it Poles of certain residual Eisenstein series of classical groups}.
Pacific J. of Math. Vol. \textbf{264} (2013), No. 1, 83--123

\bibitem[JZ]{JZ}
D. Jiang, L. Zhang,
{\it Arthur parameters and cuspidal automorphic modules of classical groups}. submitted, 2015.


\bibitem[KS03]{KS03}
H. Kim and P. Sarnak,
\textit{Refined estimates towards the Ramanujan and Selberg conjectures}.
Journal of AMS. \textbf{16} (2003), 175--181.

\bibitem[KR94]{KR94}
S. Kudla and S. Rallis,
{\it A regularized Siegel-Weil formula: the first term identity.} Ann. of Math. (2) {\bf 140} (1994), no. 1, 1--80.

\bibitem[L76]{L76}
R. Langlands,
{\it On the functional equations satisfied by Eisenstein series}. Springer Lecture Notes in Math. \textbf{544}. 1976.

\bibitem[Li89]{Li89}
J.-S. Li,
{\it Distinguished cusp forms are theta series}. Duke Math. J. {\bf 59} (1989), no. 1, 175--189.

\bibitem[Li92]{Li92}
J.-S. Li,
{\it Nonexistence of singular cusp forms}. Compositio Math. {\bf 83} (1992), no. 1, 43--51.

\bibitem[L13]{L13}
B. Liu,
{\it Fourier coefficients of automorphic forms and Arthur
classification.} Thesis (Ph.D.) University of Minnesota. 2013. 127 pp. ISBN: 978-1303-19255-5.





\bibitem[LRS99]{LRS99}
W. Luo, Z. Rudnick and P. Sarnak,
{\it On the generalized Ramanujan conjecture for $\GL(n)$.}
Proc. Sympos. Pure Math., {\bf 66}, part 2, (1999), 301--310.


\bibitem[M08]{M08}
C. M\oe glin,
{\it Formes automorphes de carr\'e int\'egrable noncuspidales}.
Manuscripta Math. \textbf{127} (2008), no. 4, 411--467.

\bibitem[M11]{M11}
C. M\oe glin,
{\it Image des op\'erateurs d'entrelacements normalis\'es et p\^oles des s\'eries d'Eisenstein}.
Adv. Math. \textbf{228} (2011), no. 2, 1068--1134.




\bibitem[MW89]{MW89}
C. M\oe glin and J.-L. Waldspurger,
{\it  Le spectre residuel de {${\rm GL}(n)$.}} Ann. Sci. \'Ecole Norm. Sup. (4) \textbf{22} (1989), no. 4, 605--674.

\bibitem[MW95]{MW95}
C. M\oe glin and J.-L. Waldspurger,
{\it Spectral decomposition and Eisenstein series.} Cambridge Tracts in Mathematics, {\bf 113}. Cambridge University Press, Cambridge, 1995.


%
%



%


\bibitem[PT11]{PT11}
O. Paniagua-Taboada,
{\it Some consequences of Arthur's conjectures for special orthogonal even groups.}
J. Reine Angew. Math. \textbf{661} (2011), 37--84.


\bibitem[PS83]{PS83}
I. Piatetski-Shapiro,
{\it On the Saito-Kurokawa lifting.}
Invent. Math. \textbf{71} (1983), no. 2, 309--338.

\bibitem[PSR88]{PSR88}
I. Piatetski-Shapiro and S. Rallis,
{\it A new way to get Euler products.}
J. reine angew. Math. {\bf 392} (1988), 110--124.

\bibitem[Sar05] {Sar05}
P. Sarnak,
{\it Notes on the generalized Ramanujan conjectures.} Harmonic analysis, the trace formula, and Shimura varieties, 659--685,
Clay Math. Proc., {\bf 4}, Amer. Math. Soc., Providence, RI, 2005.

\bibitem[Sh10]{Sh10}
F. Shahidi,
{\it Eisenstein series and automorphic L-functions,}
volume {\bf 58} of American Mathematical Society Colloquium Publications. American Mathematical Society, Providence, RI, 2010. ISBN 978-0-8218- 4989-7.


%



\bibitem[W01]{W01}
J.-L. Waldspurger,
{\it Int\'egrales orbitales nilpotentes et endoscopie pour les groupes classiques non ramifi\'es}. Ast\'erisque {\bf 269}, 2001.
\end{thebibliography}
\end{document}